\documentclass[11pt, twoside, leqno]{amsart}  
\usepackage{lipsum}
\usepackage{amsfonts}
\usepackage{graphicx}
\usepackage{epstopdf}
\usepackage{algorithmic}
\usepackage{calligra}
\usepackage{amsfonts,amsmath,amsthm,amssymb}
\usepackage{mathtools}
\usepackage{hyperref}
\usepackage{autonum}
\usepackage{hhline}
\usepackage{array}
\usepackage{diagbox}
\usepackage{tcolorbox}
\usepackage{mdframed}
\usepackage{multicol}
\usepackage{graphicx}
\usepackage{subcaption}
\usepackage{moreverb}
\usepackage{bbm}
\usepackage[margin=1.4in]{geometry}
\usepackage{todonotes}
\usepackage{scalerel,amssymb}
\allowdisplaybreaks
\usepackage{mathrsfs}  
\usepackage{lineno}
\usepackage{todonotes}
\usepackage{tikz}
\usepackage{pgfplots}
\usetikzlibrary{arrows.meta}
\usepackage{amsthm,thmtools,xcolor}
\usepackage[numbers,sort&compress]{natbib}

\newcommand{\dt}{\, \textup{d} t}
\newcommand{\ds}{\, \textup{d} s }
\newcommand{\dx}{\, \textup{d} x}

\newcommand{\dint}{\displaystyle\int}

\newcommand{\R}{\mathbb{R}} 
\newcommand{\N}{\mathbb{N}} 

\newcounter{rownumber}

\definecolor{mygreen}{HTML}{43a047}
\newcommand{\vecc}[1]{\boldsymbol{#1}}
\newtheorem{theorem}{Theorem}
\newtheorem{lemma}{Lemma}
\newtheorem{proposition}{Proposition}
\newtheorem{assumption}{Assumption}

\newtheorem{remark}{Remark}
\numberwithin{lemma}{section}
\numberwithin{proposition}{section}
\numberwithin{theorem}{section}
\numberwithin{equation}{section}
\makeatletter
\newcommand{\leqnomode}{\tagsleft@true}
\newcommand{\reqnomode}{\tagsleft@false}
\makeatother
\declaretheoremstyle[
headfont=\color{blue}\normalfont\bfseries,
bodyfont=\normalfont\itshape,
]{colored}

\title[Analysis of memory effects and thermal relaxation in sound waves]{Mathematical analysis of memory effects and thermal relaxation in nonlinear sound waves on unbounded domains}


\author[V. Nikoli\'c \& B. Said-Houari ]{\bfseries Vanja Nikoli\'c$^\ast$ and Belkacem Said-Houari}

\address{ 
Department of Mathematics \\ 
Radboud University   \\ 
Heyendaalseweg 135,
6525 AJ Nijmegen, The Netherlands}
\email{vanja.nikolic@ru.nl} 

\address{  
	Department of Mathematics\\ College of Sciences\\ University of
Sharjah, P. O. Box: 27272 \\ Sharjah, United Arab Emirates}
\email{bhouari@sharjah.ac.ae}
\thanks{$^*$Corresponding author: Vanja Nikoli\'c, \href{mailto:vanja.nikolic@ru.nl}{vanja.nikolic@ru.nl}}
\begin{document}
\vspace*{8mm}
\begin{abstract}
Motivated by the propagation of nonlinear sound waves through relaxing hereditary media,  we study a nonlocal third-order Jordan--Moore--Gibson--Thompson acoustic wave equation. Under the assumption that the relaxation kernel decays exponentially, we prove local well-posedness in unbounded two- and three-dimensional domains. In addition, we show that the solution of the three-dimensional model exists globally in time for small and smooth data, while the energy of the system decays polynomially.           
\end{abstract}
\vspace*{-7mm}
\maketitle           
    
\section{Introduction}
 Nowadays ultrasound waves are an indispensable tool in medicine, commonly used in imaging and non-invasive treatments of various disorders~\cite{duck2002nonlinear, cleveland2015nonlinear,maresca2017nonlinear,pinton2011effects}. Because of the high amplitude-to-frequency ratio that ultrasonic waves are likely to have, nonlinear effects can often be observed in their propagation. This necessitates a deeper understanding of the nonlinear acoustic models and their analytical properties. \\
   \indent Our work is particularly motivated by nonlinear sound waves in relaxing media that exhibit memory effects. These relaxation processes can occur when there are inhomogeneities in the propagation region; for example, through excitation of molecular degrees of freedom or some impurity effects in the fluid; cf.~\cite[Chapter 1]{naugolnykh2000nonlinear}. In such cases, the pressure-density state equation is not satisfied exactly but up to a term that involves the history of the process.  Additionally, classical models of nonlinear acoustics, such as the Westervelt and Kuznetsov equations, are known to exhibit parabolic-like behavior with an infinite speed of propagation; see~\cite{mizohata1993global, kaltenbacher2009global}. To avoid this paradox of diffusion, the Fourier temperature law within the governing equations (conservation of mass, momentum,  and energy) can be replaced by the Maxwell--Cattaneo law, resulting in a third-order in time acoustic wave propagation; see~\cite{jordan2008nonlinear}.\\
 \indent We investigate here such a third-order nonlinear acoustic model, known as the Jordan--Moore--Gibson--Thompson (JMGT) equation, with a memory term:
	\begin{equation}   
	\begin{aligned}
	&\tau \psi_{ttt}+\alpha \psi_{tt}-c^2\Delta \psi-b \Delta \psi_t+%
	\displaystyle\int_{0}^{t}g(s)\Delta \psi(t-s)\ds  
	=\left( k \psi_{t}^{2}\right)_t.
	\end{aligned}
	\end{equation}
\noindent The memory kernel $g$ corresponds to a particular relaxation mechanism in the medium; we refer to Section~\ref{Sec:ProblemSetting} below for more details. We are especially interested in the global solvability of the JMGT equation with memory and energy decay of its solutions in the whole $\R^3$. This type of memory acting only on the solution of the equation is often referred to as memory type I; cf.~\cite{lasiecka2017global, dell2016moore}.\\
\indent In the present work, we study the influence of memory of type I on the behavior of solutions in media whose parameters satisfy the so-called subcritical condition: $\alpha b- \tau c^2>0$. In smooth bounded domains, it is known that the exponential decay of the relaxation kernel $g$ directly influences how the energy of the system decays; see, for example,~\cite{lasiecka2017global}. The situation in the whole space $\R^n$ is different. As it turns out, although our memory kernel decays exponentially, the solution decays polynomially at most. It has been proven that in the absence of the memory term (i.e., when $g=0$),  the same polynomial decay rate is optimal for the linearized problem; see~\cite[Theorem 3.6]  {PellSaid_2019_1}.  We thus expect the decay rate to be sharp even with $g>0$ since the damping induced by the memory term is weak; see Theorem~\ref{Theorem_Decay} below and the discussion directly after it.      \\         
 \indent We organize the paper as follows.  We begin by discussing the modeling aspects and setting our problem in Section~\ref{Sec:ProblemSetting}. Section~\ref{Sec:Preliminaries} contains the necessary theoretical preliminaries, which allow us to rewrite the equation with the corresponding initial data as an abstract first-order Cauchy problem. In Section~\ref{Sec:EnergyEstimates}, we derive several energy estimates that are uniform in time and crucial for establishing global existence. Section~\ref{Sec:LocalExistence} is dedicated to proving well-posedness of the problem in $\R^2$ and $\R^3$ for sufficiently short final time. In Section~\ref{Sec:GlobalExistence}, we prove that in $\R^3$ the solution exists globally in time provided that the data are sufficiently small. Finally, in Section~\ref{Sec:DecayRates}, we show that in $\R^3$ the energy of the system decays polynomially with time. 
     \\  
 \indent The main results of our work concerning local well-posedness are contained in Theorems~\ref{Local_Existence_1} and \ref{Thm:LocalExistence_3D}. The global existence of solutions of the JMGT equation with memory in $\R^3$
  for small data is obtained in Theorem~\ref{Thm:GlobalExistence}, whereas the polynomial energy decay is established in Theorem~\ref{Theorem_Decay}.
\section{Problem setting and modeling} \label{Sec:ProblemSetting}
In nonlinear acoustics, the Kuznetsov equation is one of the classical models. It is given by
\begin{equation}\label{Kuznt}
\psi_{tt}-c^{2}\Delta \psi-\delta \Delta \psi_{t}=\left( \tfrac{1}{c^{2}}\tfrac{B}{2A}(\psi_{t})^{2}+|\nabla \psi|^{2}\right)_t,
\end{equation}
where $\psi=\psi(x,t)$ represents the acoustic velocity potential for $x \in \R^n$ and $t>0$; see~\cite{kuznetsov1971equations}. The equation \eqref{Kuznt} can be obtained as an approximation of the governing equations of fluid mechanics by means of asymptotic expansions in powers of small parameters; see~\cite{crighton1979model, kuznetsov1971equations, kaltenbacher2007numerical}.
The constants $c>0$ and $\delta >0$ are the speed and the diffusivity of sound, respectively. The parameter of nonlinearity $B/A$ arises in the Taylor expansion of the variations of pressure in a medium in terms of the variations of density; cf.~\cite{beyer1960parameter}. Typical values of these parameters in different media can be found in, e.g.,~\cite{kaltenbacher2007numerical, naugolnykh2000nonlinear}. If we can neglect local nonlinear effects and assume 
\begin{equation} \label{local_approximation}
|\nabla \psi|^2 \approx \frac{1}{c^2}\psi_t^2,
\end{equation} 
we arrive at the Westervelt equation in the potential form
 \begin{equation}\label{Westervelt}
 \psi_{tt}-c^{2}\Delta \psi-\delta \Delta \psi_{t}=\left( \tfrac{1}{c^{2}}(\tfrac{B}{2A}+1)(\psi_{t})^{2}\right)_t;
  \end{equation}
cf.~\cite{westervelt1963parametric}.  After solving equation \eqref{Kuznt} or \eqref{Westervelt} for the acoustic velocity potential, the acoustic pressure can be computed as $u= \varrho \psi_t$, where $\varrho$ denotes the mass density of the medium. \\
\indent In the derivation of these models, the classical Fourier law of heat conduction is used in the equation for the conservation of energy. It is, however, well-known that the Fourier law predicts an infinite speed of heat propagation: any thermal disturbance at one point has an instantaneous effect elsewhere
in the medium~\cite{liu1979instantaneous}. To overcome this drawback, the Maxwell--Cattaneo law can be used instead. Introducing this law of heat conduction in the derivation of \eqref{Westervelt} leads to a third-order equation given by
\begin{equation}\label{JMGT}
\tau \psi_{ttt}+\psi_{tt}-c^{2}\Delta \psi-b\Delta \psi_{t}= \left( \tfrac{1}{c^{2}}(\tfrac{B}{2A}+1)(\psi_{t})^{2}\right)_t;
\end{equation}
cf.~\cite{jordan2008nonlinear}. This nonlinear equation is often referred to as the  Jordan--Moore--Gibson--Thompson (JMGT) equation. Here $\tau>0$ stands for the relaxation time. The constant $b>0$ is given by 
\begin{equation} \label{b}
b=\delta +\tau c^2.
\end{equation}
\indent Additionally, it is well-known that relaxation processes play an important role in high-frequency waves in fluids and gases. If relaxation occurs, acoustic pressure can depend on the density at all prior times. Such a process, therefore, introduces a memory term into the state equation.
This motivates us to consider the general nonlocal JMGT equation in the form of
\begin{equation}   \label{Main_Equation}
\begin{aligned}
&\tau \psi_{ttt}+\alpha \psi_{tt}-c^2\Delta \psi-b \Delta \psi_t+%
\displaystyle\int_{0}^{t}g(s)\Delta \psi(t-s)\ds
=\left( k \psi_{t}^{2}\right)_t.
\end{aligned}
\end{equation}
The function $g$ denotes the relaxation memory kernel related to the particular relaxation mechanism. In  \eqref{Main_Equation}, we have introduced $k=\tfrac{1}{c^{2}}(\tfrac{B}{2A}+1)$. The coefficient $\alpha>0$ accounts for the losses due to friction. Equation~\eqref{Main_Equation} is supplemented with the following initial data: 
\begin{equation}  \label{initial data}
\psi(x,0)=\psi_{0}(x),\qquad \psi_{t}(x,0)=\psi_{1}(x), \qquad \psi_{tt}(x,0)=\psi_{2}(x),
\end{equation}
whose regularity will be specified in the theorems below. We note that a version of the non-local JGMT equation with a quadratic gradient nonlinearity could also be considered analogous to the Kuznetsov equation. To handle such nonlinearity in the existence analysis, typically, solution spaces of higher regularity are needed than those in the present work; see~\cite{mizohata1993global, KaltenbacherNikolic} for the analysis on bounded domains with $H^3$-regular solutions. We thus restrict our considerations here to equation \eqref{Main_Equation}.
\subsection{Memory kernel}
Throughout the paper, we make the following assumptions on the relaxation kernel; cf.~\cite[Section 1]{dell2016moore}. 
\begin{assumption} The memory kernel is assumed to satisfy the following conditions:
\begin{enumerate}
	\item[(G1)] \label{itm:first}  $g\in W^{1,1}(\R^+)$ and $g'$ is almost continuous on $\R^+=(0, +\infty)$. \vspace{0.1 cm}
	\item[(G2)] $g(s) \geq 0$ for all $s>0$ and 
	\reqnomode
	\begin{align} \label{def_cg}
	\ c^2_g:=c^2-\displaystyle\int_{0}^{\infty}g(s)\ds>0. 	\vspace{0.1 cm}
	\end{align}
	\item[(G3)] There exists $\zeta>0$, such that the function $g$ satisfies the 
	differential inequality given by
	\begin{equation}
	g^\prime(s)\leq -\zeta g(s)
	\end{equation}
	 for every $s\in (0,\infty)$. 	\vspace{0.1 cm}
	\item[(G4)] It holds that $g^{\prime\prime}\geq0$ almost everywhere.\vspace{0.2 cm}
\end{enumerate}
\end{assumption}
\noindent In relaxing media, the memory kernel typically has the exponential form
\begin{equation}
g(s)=m c^2\exp{(-s/\tau)},
\end{equation}
where $m$ is the relaxation parameter; see~\cite[Chapter 1]{naugolnykh2000nonlinear} and \cite[Section 1]{lasiecka2017global}. The value of $m$ is small, so the condition \eqref{def_cg}, equivalent to $m< \tau$, is easily satisfied. With this choice of the kernel, we have 
\[g'(s) \leq -g(s);\]
i.e., we can take $\zeta=1$. We see also that for $\tau \rightarrow 0^+$, the kernel tends to zero and we are formally in the regime of the Westervelt equation, as expected.\\
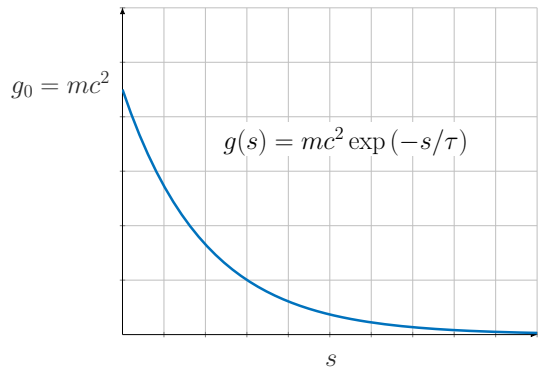
\begin{figure}[h]
\hspace*{-1.8cm} \definecolor{mycolor1}{rgb}{0.00000,0.44700,0.74100}%
\begin{tikzpicture}[scale=0.48, font=\huge]

\begin{axis}[%
width=4.521in,
height=3.566in,
at={(0.758in,0.481in)},
scale only axis,
axis x line=center,
axis y line=center,
axis line style = {-Latex},
xmin=0,
xmax=5,
xlabel style={at={(0.53, -0.11)}, font=\color{white!15!black}}, xticklabels={,,,,,,},
xlabel={\huge $s$},
ymin=0,
ymax=0.3,
ylabel style={at={(-0.28,0.82)}, font=\color{white!15!black}}, yticklabels={, , , , , ,, },
ylabel={\huge $g_0= mc^2$},
axis background/.style={fill=white},
xmajorgrids,
ymajorgrids,
legend style={legend cell align=left, align=left, draw=none,  empty legend, at={(0.85,0.65)}}
]
\addplot [color=mycolor1, line width=2pt]
  table[row sep=crcr]{%
0	0.225\\
0.102040816326531	0.203173356166206\\
0.204081632653061	0.183464056248179\\
0.306122448979592	0.165666702416925\\
0.408163265306123	0.149595821933489\\
0.510204081632653	0.135083934269641\\
0.612244897959184	0.121979805731992\\
0.714285714285714	0.110146873400314\\
0.816326530612245	0.0994618219553605\\
0.918367346938775	0.0898132985647839\\
1.02040816326531	0.0811007524345095\\
1.12244897959184	0.0732333869321069\\
1.22448979591837	0.0661292133618421\\
1.3265306122449	0.0597141965304742\\
1.42857142857143	0.0539214831993995\\
1.53061224489796	0.0486907053825858\\
1.63265306122449	0.0439673512296883\\
1.73469387755102	0.0397021969381077\\
1.83673469387755	0.0358507937737196\\
1.93877551020408	0.0323730048543522\\
2.04081632653061	0.0292325868686385\\
2.14285714285714	0.0263968123712068\\
2.24489795918367	0.0238361287179893\\
2.3469387755102	0.0215238500873047\\
2.44897959183673	0.0194358793771379\\
2.55102040816327	0.0175504570804206\\
2.6530612244898	0.0158479345212452\\
2.75510204081633	0.0143105690888161\\
2.85714285714286	0.0129223393352138\\
2.95918367346939	0.0116687780100175\\
3.16326530612245	0.00951467264523309\\
3.36734693877551	0.0077582241660803\\
3.57142857142857	0.00632602344351874\\
3.77551020408163	0.0051582129816401\\
3.97959183673469	0.00420598522934945\\
4.18367346938776	0.00342954271420659\\
4.48979591836735	0.00252516014338067\\
4.79591836734694	0.00185926646234869\\
5	0.0015160380747945\\
};
\addlegendentry{$g(s)=mc^2 \exp{(-s/\tau)}$}
\end{axis}
\end{tikzpicture}%
\caption{The fading relaxation kernel}	
\end{figure}
\subsection{Previous work}
The JMGT equation and its linearization have been a subject of extensive study. The linearization of this equation without memory is given by
\begin{equation}\label{MGT_2}
\tau \psi_{ttt}+\alpha\psi_{tt}-c^{2}\Delta \psi-b \Delta \psi_{t}=0. 
\end{equation}
This equation is known as the Moore--Gibson--Thompson equation, although, as mentioned in \cite{bucci2019feedback}, this model originally appears in the work of Stokes~\cite{stokes1851examination}.	Interestingly, equation \eqref{MGT_2} also arises in viscoelasticity theory under the name of \emph{standard linear model} of vicoelasticity;~see \cite{Gorain_2010} and references given therein. \\
\indent If $b=0$ in \eqref{MGT_2}, there is no semigroup associated with the linear dynamics; see~\cite{Kaltenbacher_2011}. For $b>0$, the linear dynamics is described by a strongly continuous semigroup, which is exponentially stable if 
	\begin{eqnarray}\label{Condition_Parameters}
	\alpha b-\tau c^2>0. 
	\end{eqnarray}
If $\alpha b=\tau c^2$, the energy is conserved. \\
\indent	 The linear model associated with the JMGT equation with memory \eqref{Main_Equation} in the pressure form reads as
\begin{equation}  \label{Main_Equation_Linear}
\begin{aligned}
&\tau u_{ttt}+\alpha u_{tt}-c^2\Delta u-b \Delta u_t+%
\displaystyle\int_{0}^{t}g(s)\Delta u(t-s)\ds
=0.
\end{aligned}
\end{equation} 
Recall that the pressure and potential are connected via $u= \varrho \psi_t$. In~\cite{lasiecka2016moore}, a generalization of this equation is studied in smooth bounded domains with a memory term in the form of $\int_{0}^{t}g(s)\Delta z(t-s)\, \textup{d}s,$ where $z$ is one of the three functions:  $z=u, z=u_{t}$, or  $z= u+u_{t}$. If the memory kernel $g$ decays exponentially, the same holds for the solution, provided that the critical condition \eqref{Condition_Parameters} holds. This result is extended in~\cite{lasiecka2015moore} by allowing the memory kernel to satisfy a more general decay property. \\
\indent The critical case where $\alpha b =\tau c^2$ and $\int_{0}^{\infty}g(s)\ds>0$ is investigated in~\cite{dell2016moore} with  a general strictly positive self-adjoint linear operator $A$ instead of $-\Delta$. The linearized problem  is exponentially stable if and only if $A$ is a bounded operator.  In the case of an unbounded operator $A$, the corresponding energy decays polynomially with the rate $1/t$ for regular initial data.\\
\indent Taking the quadratic nonlinear effects into account leads to the nonlinear JMGT equation of Westervelt type given in \eqref{JMGT}. Without memory effects, it is analyzed in~\cite{KaltenbacherNikolic} in terms of existence and regularity of solutions on bounded smooth domains. Furthermore, it is shown that its solution converges weakly to the solution of the Westervelt equation in the limit $\tau \rightarrow 0^+$.\\
\indent The JMGT equation with memory is investigated in~\cite{lasiecka2017global} on regular bounded domains, expressed in terms of the acoustic pressure $u$.  There it is proven that with suitable adjustment of the memory kernel, solutions exist globally for sufficiently small and regular initial data. With exponentially
decaying memory kernel these solutions exhibit exponential decay rates. \\
\indent	Due to the lack of Poincar\'e's inequality, the analysis of nonlinear acoustic models is more delicate in $\R^n$. Nevertheless, the linearized problem \eqref{Main_Equation_Linear} with and without memory is well-understood; see~\cite{PellSaid_2019_1, Bounadja_Said_2019}. The nonlinear JMGT equation \eqref{JMGT} is also known to have solutions globally in time in $\R^3$ in non-hereditary media; cf.~\cite{Racke_Said_2019}. 
\section{Theoretical preliminaries and notation}\label{Sec:Preliminaries}
We collect here several theoretical results that will be helpful in the later proofs.
\subsection{The history framework} Following~\cite{dell2016moore}, we use the history framework of Dafermos~\cite{dafermos1970asymptotic} to transform our problem into an evolution one. We introduce the auxiliary history variable $\eta=\eta(x, t,s)$, defined as
\begin{equation}  \label{def of eta}
\eta(x, t, s)= \begin{cases}
\psi(x, t)-\psi(x, t-s), \quad  &0<s\leq t, \\
\psi(x, t), \quad &s>t.
\end{cases}
\end{equation}
The JMGT equation \eqref{Main_Equation} then transforms into the following problem:
\begin{equation}  \label{eta syst}
\begin{cases}
\tau \psi_{ttt}+\alpha \psi_{tt}-b \Delta \psi_{t}-c^2_g\Delta \psi-%
\displaystyle\int_{0}^{\infty}g(s)\Delta \eta(s)\ds 
= 2k\psi_{t}\psi_{tt},  \\[2mm] 
\eta_{t}(x,s)+\eta_{s}(x,s)=\psi_{t}(x,t),%
\end{cases}%
\end{equation}
where we recall that the modified speed of sound squared $c^2_g$ is defined in \eqref{def_cg}. The problem is supplemented with the initial data \eqref{initial data}. \\
\indent Note that from the second equation in \eqref{eta syst} we get \eqref{def of eta} via Duhamel's formula, if we additionally set 
\begin{equation}
\eta(t=0)=\psi_0 \quad  \text{and} \quad \eta(s=0)=0;
\end{equation} see~\cite[Section 3]{grasselli2002uniform} for a detailed discussion on this supplementary equation. Therefore, we can obtain equation \eqref{Main_Equation} from \eqref{eta syst}. Indeed, it is enough to check that
\begin{equation}
\begin{aligned}
\int_{0}^{\infty}g(s)\Delta \eta(s)\ds =&\, \int_{0}^{t}g(s)\Delta \eta(s)\ds +\int_{t}^{\infty}g(s)\Delta \eta(s)\ds \\
=&\, \int_{0}^{t}g(s)\Delta (\psi(t)-\psi(t-s))\ds +\int_{t}^{\infty}g(s)\Delta \psi(t)\ds\\
=&\, (c^2-c^2_g) \Delta \psi-\int_0^t g(s) \Delta \psi(t-s) \ds.
\end{aligned}
\end{equation}
\subsection{\bf Setting $\boldsymbol{\alpha=1}$} From this point on, we set $\alpha=1$. We may do so without the loss of generality since we can always re-scale other coefficients in the equation. The subcritical condition \eqref{Condition_Parameters} then reads as 
\begin{equation}\label{critical_condition_scaled}
b >\tau c^2,
\end{equation}
which, having in mind relation \eqref{b}, means that we need the sound diffusivity $\delta$ to be positive. In other words, we are assuming our medium to be thermoviscous. For our well-posedness result, we will also require that $\tau c^2 > \tau c^2_g$, which is equivalent to assuming that $\int_{0}^{\infty}g(s)\ds>0$.
\subsection{Functional spaces} For future use, we introduce here the weighted $L^{2}$ spaces,
\begin{equation}
L^2_{\tilde{g}}=L^{2}_{\tilde{g}}(\mathbb{R}^{+}, L^2(\mathbb{R}^{n})).
\end{equation}
We will have three types of weights: $\tilde{g} \in \{g, -g' , g''\}$. The space is endowed with the inner product 
\begin{equation}
\left(\eta,\tilde{\eta} \right)_{L^2, \tilde{g}}=%
\displaystyle\int_{0}^{t}\tilde{g}(s)\left(
\eta(s),\tilde{\eta}(s)\right)_{L^{2}(\mathbb{R}%
	^{n})} \ds
\end{equation}
for $\eta, \tilde{\eta}\in L^2_{\tilde{g}}$, and with the following norm:
\begin{equation}
\Vert\eta\Vert^{2}_{L^2, \tilde{g}}=\int_{0}^{t}\tilde{g}(s)\Vert%
\eta(s)\Vert_{L^{2}}^{2}\ds.
\end{equation}   
We can then further introduce the spaces
\[H^{m}_{\tilde{g}}=\{\eta \in L^2_{\tilde{g}}:\ D^{\alpha} \eta \in L^2_{\tilde{g}}, \quad \forall \alpha: |\alpha| \leq m\}, \quad m \in \{1, 2\}.\]
The infinitesimal generator of the right-translation $C_{0}$%
-semigroup on $L^2_{\tilde{g}}$ is given by the linear operator $\mathbb{T}$:
\begin{equation} \label{def_T_eta}
\mathbb{T}\eta=-\eta_s\quad \text{with}\quad D(\mathbb{T})=\{\eta\in%
L^2_{\tilde{g}}\, | \ \eta_s\in L^2_{\tilde{g}},\ \eta(s=0)=0\},
\end{equation}
where the index $s$ denotes the distributional derivative with respect to the
variable $s>0$; cf.~\cite{dell2016moore, Bounadja_Said_2019}.
\subsection{Auxiliary inequalities}
Throughout the paper, we often use the Ladyzhenskaya inequality for functions $f \in H^1(\R^n)$, with $ n \in \{2, 3\}$, given by 
\begin{eqnarray}  \label{Ladyz_Ineq}
\Vert f\Vert_{L^4}\leq 
C_n\Vert f\Vert_{L^2}^{1-n/4}\Vert \nabla f\Vert_{L^2}^{n/4},
\end{eqnarray}
where the constant $C_n>0$ depends on $n$. Furthermore, we frequently rely also on the particular case of the Gagliardo--Nirenberg interpolation inequality~\cite{Ner59,nirenberg2011elliptic}:
\begin{equation} \label{Interpolation_inequality}
\begin{aligned}
\left\Vert \nabla u\right\Vert _{L^{4}}\leq C_n \left\Vert \nabla u\right\Vert
_{L^{2}}^{1-n/4}\left\Vert
\nabla^{2}u\right\Vert _{L^{2}}^{n/4}.  
\end{aligned}
\end{equation}
We need the following estimate as well:
\begin{equation}
\Vert \nabla( uv) \Vert _{L^2}\leq C( \Vert u\Vert _{L^{\infty}}\Vert
\nabla v\Vert _{L^{2}}+\Vert v\Vert _{L^{4}}\Vert \nabla u\Vert
_{L^{4}}).
\label{First_inequaliy_Guass}
\end{equation}
\noindent The next technical estimate will be employed when deriving the decay rate of the energy of our system.
\begin{lemma}[see Lemma 3.5 in~\cite{PellSaid_2019_1}] \label{Lemma:Ineq}
   	Let $n\geq 1$ and $t\geq 0 $. Then the following estimate holds:
   	\begin{equation}\label{eq of DR}
   	\int_{0}^{1}r^{n-1}e^{-r^{2}t}\textup{d}r \leq C(n)(1+t)^{-{n}/{2}}.
\end{equation}    
\end{lemma}
\noindent We state here one more useful inequality that will be crucial in our energy arguments.
\begin{lemma}[see Lemma 3.7 in~\cite{strauss1968decay}] 
\label{Lemma_Stauss} Let $M=M(t)$ be a non-negative continuous function
satisfying the inequality
\begin{equation}
M(t)\leq C_1+C_2 M(t)^{\kappa},
\end{equation}
in some interval containing $0$, where $C_1$ and $C_2$ are positive
constants and $\kappa>1$. If $M(0)\leq C_1$ and
\begin{equation}
C_1C_2^{1/(\kappa-1)}<(1-1/\kappa)\kappa^{-1/(\kappa-1)},
\end{equation}
then in the same interval
\begin{equation}
M(t)<\frac{C_1}{1-1/\kappa}.
\end{equation}
\end{lemma}
\subsection{Notation} Throughout the paper, the constant $C$ denotes a generic positive constant that does not depend on time, and that may take different values of different occasions. We use $x \lesssim y$ to denote $x \leq C y$.
\section{A priori energy estimates} \label{Sec:EnergyEstimates}
In this section, we formally derive several energy estimates for our problem that that will be utilized later. We begin by rewriting our equation \eqref{eta syst} as a first-order in time system. To this end, we introduce the functions 
\begin{equation}\label{Change_Variable}
v=\psi_{t}\quad \text{and}\quad w=\psi_{tt},
\end{equation}
which leads to the following system of equations:
\begin{equation}  \label{Main_System}
\begin{cases}
\psi_{t}=v, \\ 
v_{t}=w, \\ 
\tau w_{t}=- w+c^2_g\Delta \psi+b\Delta v + \displaystyle%
\int_{0}^{\infty}g(s)\Delta\eta(s)\ds+2k vw,
\\ 
\eta_{t}=v-\eta_{s},%
\end{cases}%
\end{equation}
with the initial data 
\begin{equation}  \label{Main_System_IC}
(\psi, v, w, \eta) |_{t=0} =(\psi_0, \psi_1, \psi_2, \psi_0).
\end{equation}
By using the notation
\[\vecc{\Psi}=(\psi, v, w, \eta)^T,\]
and setting $\vecc{\Psi}_0=\Psi(0)$, we can convert our problem into an initial value problem for a first-order abstract evolution equation. Indeed, $\Psi$ satisfies
\begin{equation} \label{abstract_evol_eq}
\begin{aligned}
\begin{cases}
\dfrac{\textup{d}}{\dt}\vecc{\Psi}(t)= \mathcal{A} \vecc{\Psi}(t)+\mathcal{F}(\vecc{\Psi}),\quad t>0, \vspace{0.2cm}\\[1mm] 
\vecc{\Psi}(0)=\vecc{\Psi}_0,
\end{cases} 
\end{aligned}
\end{equation}
where the operator $ \mathcal{A}$ is defined as
\begin{equation}
\begin{aligned}
\mathcal{A}\begin{bmatrix}
\psi \\[1mm]
v \\[1mm]
w \\[3mm]
\eta^t 
\end{bmatrix} = \begin{bmatrix}
v \\[1mm]
w \\[1mm]
-\dfrac{1}{\tau} w+ \dfrac{c^2_g}{\tau} \Delta \psi+\dfrac{b}{\tau} \Delta v+\dfrac{1}{\tau}\displaystyle \int_0^\infty g(s)\Delta \eta(s) \ds \\[3mm]
v+\mathbb{T} \eta
\end{bmatrix},
\end{aligned}
\end{equation}
with the operator $\mathbb{T}$ as in \eqref{def_T_eta}. The nonlinear term in \eqref{abstract_evol_eq} is given by
\begin{equation} \label{def_F}
\begin{aligned}
\mathcal{F}(\vecc{\Psi})= \dfrac{2k}{\tau} \, [
0,\
0 ,\
v w ,\
0 
]^T.
\end{aligned}
\end{equation}
Going forward, our work plan is to introduce the mapping 
\begin{equation}
\begin{aligned}\label{Mapping_T}
\mathcal{T}(\vecc{\Phi})=\vecc{\Psi},
\end{aligned}
\end{equation}
where $\vecc{\Psi}$ solves the inhomogeneous linear problem
\begin{equation} 
\begin{cases}
\partial_t \vecc{\Psi}-\mathcal{A} \vecc{\Psi} = \mathcal{F}(\vecc{\Phi}), \\
\vecc{\Psi}_{t=0}=\vecc{\Psi}_0
\end{cases}
\end{equation}
on a suitably defined ball in a Banach space and employ the contraction principle on $\mathcal{T}$. The unique fixed-point is then the solution to our nonlinear problem. \\
\indent In preparation, we first derive several energy estimates for problem \eqref{Main_System}, which are uniform in time. They will be crucial in proving global solvability.
\subsection{Functional setting} In order to formulate our results, we introduce the Hilbert spaces
  \begin{equation} \label{Hilbert_spaces}
\begin{aligned}
& \mathcal{H}^{s+1}= H^{s+1}(\R^n) \times H^{s+1}(\R^n) \times H^{s}(\R^n) \times  H^{s+1}_{-g'}(\R^n), \\
\end{aligned}
\end{equation}
for $s \in \{0, 1\}$ and  $n \in \{2,3\}$. It is known that the  homogeneous Sobolev space $\dot{H}^{1}(\R^n)$ is a Hilbert space if and only if $n>2$; see~\cite[Proposition 1.34]{bahouri2011fourier}. For $n=3$, we can therefore work with the Hilbert spaces
  \begin{equation} \label{dotHilbert_space_1}
\begin{aligned}
 \dot{\mathcal{H}}^{1}=&\, \dot{H}^{1}(\R^3) \times H^{1}(\R^3) \times L^2(\R^3) \times  \dot{H}^{1}_{-g'}(\R^3),
 \end{aligned}
 \end{equation}
where
$ \dot{H}^{1}_{-g'}(\R^3)= \{\eta: \nabla \eta \in L^2_{-g'}\},$ as well as
   \begin{equation} \label{dotHilbert_space_2}
 \begin{aligned}
 \dot{\mathcal{H}}^{2}=&\, \begin{multlined}[t] \{ \psi \in \dot{H}^{1}(\R^3): \, \nabla^2 \psi \in L^2(\R^3)\} \times H^{2}(\R^3) \times H^1(\R^3)\\
  \times \{\eta \in \dot{H}^{1}_{-g'}(\R^3):\, \nabla^2 \eta \in L^2_{-g'}(\R^3) \}.  \end{multlined}
\end{aligned}
\end{equation}
We intend to work with these spaces to show global well-posedness in $\R^3$. \\[2mm]
\noindent \textbf{Energy functionals.} We then define the energy of first order by
\begin{equation}\label{E_1_Eqv}
\begin{aligned}
\mathscr{E}_1[\vecc{\Psi}] =&\, \begin{multlined}[t]\Vert\nabla(\psi +\tau v)\Vert^{2}_{L^2} + \Vert v+\tau w\Vert^{2}_{L^2}+\Vert \nabla v\Vert^{2}_{L^2}
+\Vert \nabla\eta\Vert^{2}_{L^2, -g'}. 
\end{multlined}
\end{aligned}
\end{equation} 
 We also introduce the  energy functional of second order as follows:
\begin{equation}\label{E_2_Eqv}
\begin{aligned}
\mathscr{E}_2[\vecc{\Psi}] =&\, \begin{multlined}[t] \Vert\Delta(\psi +\tau v)\Vert^{2}_{L^2} + \Vert\nabla(v+\tau w)\Vert^{2}_{L^2}+\Vert \Delta v\Vert^{2}_{L^2}
+\Vert \Delta\eta\Vert^{2}_{L^2, -g'}. 
\end{multlined}
\end{aligned}
\end{equation}
Thus for $\vecc{\Psi} \in \dot{\mathcal{H}}^2$, we have the norm
\begin{equation}
\|\vecc{\Psi}\|_{\dot{\mathcal{H}}^2}= \left(\mathscr{E}_1[\vecc{\Psi}]+\mathscr{E}_2[\vecc{\Psi}]+\Vert w\Vert _{L^{2}}^{2} \right)^{1/2},
\end{equation}  
whereas for $\vecc{\Psi} \in \mathcal{H}^2$, the norm is given by
\begin{equation}
\|\vecc{\Psi}\|_{\mathcal{H}^2}= \left(\|\psi\|^2_{L^2}+\mathscr{E}_1[\Psi]+\mathscr{E}_2[\vecc{\Psi}]+\Vert w\Vert _{L^{2}}^{2} \right)^{1/2}.
\end{equation}
For $\Psi \in C([0,T]; \mathcal{H}^2)$,  we can introduce here the energy semi-norm by
\begin{equation} \label{EnergyNorm}
\begin{aligned}
|\vecc{\Psi}|_{\mathcal{E}(t)}
= \, \sup_{0\leq \sigma \leq t}  \left(\mathscr{E}_1[\vecc{\Psi}](\sigma)+\mathscr{E}_2[\vecc{\Psi}](\sigma)+\Vert w(\sigma )\Vert _{L^{2}}^{2} \right)^{1/2} . 
\end{aligned}
\end{equation}
The corresponding dissipation semi-norm is given by
\begin{equation} \label{DissipativeEnergyNorm}
\begin{aligned}
&|\vecc{\Psi}|_{\mathcal{D}(t)}\\
=&\, \left\{ \vphantom{\int_{0}^{t}} \right.\begin{multlined}[t]\int_{0}^{t}\left(\Vert \nabla
v(\sigma )\Vert _{L^2}^{2}+\Vert \nabla \eta (\sigma )\Vert _{L^2, -g'}^{2}+\mathscr{E}_2[\vecc{\Psi}](\sigma)\right.  \left. \left.
+\Vert w(\sigma )\Vert
_{L^{2}}^{2} \right) \textup{d} \sigma  \vphantom{\int_{0}^{t}}\right\}^{1/2}.  \end{multlined}
\end{aligned}
\end{equation}%
We can easily see here one of the difficulties in the analysis of the JMGT equation in $\R^n$, which is that in general, we do not have direct control over $\|\psi(t)\|_{L^2}$ because of the lack of Poincar\'e's inequality.
\subsection{Derivation of the estimates}  We derive the energy estimates under the assumption that a sufficiently smooth solution $\vecc{\Psi}=(\psi, v, w, \eta)^T$ of our system \eqref{Main_System} with initial conditions  \eqref{Main_System_IC} exists on some time interval $[0, T]$. In particular, we assume that 
$|\vecc{\Psi}|_{\mathcal{E}(T)}< \infty$.  The estimates below will then be rigorously justified in Section~\ref{Sec:LocalExistence}.\\
\indent  To simplify the notation that involves the nonlinear term $2kvw$ in the system, we also introduce the functionals $R^{(1)}$ and $R^{(2)}$ as
\begin{equation} \label{Def_R}
\begin{aligned}
R^{(1)}(\varphi)= 2k(vw, \varphi)_{L^2}, \quad R^{(2)}(\varphi)= 2k(\nabla(vw), \nabla \varphi)_{L^2},
\end{aligned}
\end{equation}
where $\varphi$ stands for various test functions that we use in the proofs.\\
\indent Our main goal now is to derive an estimate in the form
\begin{equation}\label{Main_Estimate_Bootstrap}
\begin{aligned}
|\vecc{\Psi}|^2_{\mathcal{E}(t)}+|\vecc{\Psi}|^2_{\mathcal{D}(t)} \lesssim&\, |\vecc{\Psi}|^2_{\mathcal{E}(0)}+\sum_i \int_0^t |R^{(1)}(\varphi_i) |\, \textup{d}\sigma+\sum_j \int_0^t |R^{(2)}(\varphi_j) |\, \textup{d}\sigma \\
\lesssim&\,  |\vecc{\Psi}|^2_{\mathcal{E}(0)}+ |\vecc{\Psi}|_{\mathcal{E}(t)}|\vecc{\Psi}|^2_{\mathcal{D}(t)}
\end{aligned}
\end{equation}
for all $t \in [0,T]$. On account of Lemma~\ref{Lemma_Stauss}, this inequality together with a bootstrap argument yields
\begin{eqnarray*}
|\vecc{\Psi}|_{\mathcal{E}(t)}+|\vecc{\Psi}|_{\mathcal{D}(t)} \lesssim |\vecc{\Psi}|_{\mathcal{E}(0)} 
\end{eqnarray*}
provided that $|\vecc{\Psi}|_{\mathcal{E}(0)}$ is small enough. The hidden constant does not depend on time, and so the above estimates allow us to continue the solution to $T=\infty$. 

\subsection{Lower-order estimates}
In order to formulate our results and following~\cite{dell2016moore}, we introduce here the weighted lower-order energy of first order at time $t \geq 0$ as
\begin{equation} \label{energy}
\begin{aligned}
 E_1(t)
=& \,\begin{multlined}[t] \dfrac{1}{2}\left [\vphantom{\int_{0}^{%
		\infty}} c^2_g%
\|\nabla (\psi+\tau v)\|_{L^2}^{2}+\tau(b -\tau
c^2_g )\|\nabla v\|^{2}_{L^2}+\|v+\tau w\|^{2}_{L^2}\right.   \\
 \left. +\tau \Vert \nabla
 \eta \Vert _{L^2, -g'}^{2}+ \|\nabla \eta\|^2_{L^2, g}
\right.
 \left.+ 2\tau \int_{\mathbb{R}^{n}}
\int_{0}^{\infty}g(s) \nabla \eta (s) \cdot \nabla v
\ds \dx\right]. \end{multlined}
\end{aligned}
\end{equation}
We remark that the last term in \eqref{energy} has an undefined sign, but we will show that the other terms in the energy functional can absorb it. In fact, $E_1$ is equivalent to the energy $\mathscr{E}_1=\mathscr{E}_1[\vecc{\Psi}]$, introduced in \eqref{E_1_Eqv}.
\begin{lemma} \label{Lemma_EquivE_0}
	Assume that $b  \geq \tau c^2 >\tau c^2_g.$ There exist positive constants $C_{1}$ and $C_{2}$, such that
	\begin{equation}\label{equiv E and E0}
	C_{1} \mathscr{E}_1(t)\leq E_1(t)\leq C_{2} \mathscr{E}_1(t), 
	\end{equation}
	for all $t\geq0$.
\end{lemma}
\begin{proof}
	The proof follows analogously to the proof of~\cite[Lemma 3.1]{dell2016moore}; we include it here for completeness. To show \eqref{equiv E and E0}, we first have by Young's inequality  
	\begin{equation}
	\begin{aligned}
	\left|2\tau\int_{\R^n}\int_{0}^\infty g(s)\nabla \eta(s) \cdot \nabla v \ds\dx \right|
	\leq& \, \frac{\tau^{2}(c^2-c^2_g)}{\varepsilon+1}\Vert\nabla v\Vert^{2}_{L^2}+(\varepsilon+1)\dint_{0}^\infty g(s)\Vert\nabla\eta(s)\Vert^{2}_{L^2}\ds.
	\end{aligned}
	\end{equation}
	for every $\varepsilon>0$. By using assumption (G3) on the relaxation kernel $g$, we then have
		\begin{equation}
	\begin{aligned}
2\tau\int_{\R^n}\int_{0}^\infty g(s)\nabla \eta(s) \cdot \nabla v \ds\dx 
	\geq& \, \begin{multlined}[t]-\frac{\tau^{2}(c^2-c^2_g)}{\varepsilon+1}\Vert\nabla v\Vert^{2}_{L^2}- \|\nabla \eta\|^2_{L^2, g} \\
	-\frac{\varepsilon}{\zeta}\dint_{0}^\infty(- g'(s))\Vert\nabla\eta(s)\Vert^{2}_{L^2}\ds.\end{multlined}
	\end{aligned}
	\end{equation}
	Since $b-\tau c^2_g \, \geq\, \tau (c^2-c^2_g)$, by reducing $\varepsilon$, we obtain
	\begin{equation} \label{left-hand side}  
	\begin{aligned}
	E_1(t)\geq&\, \begin{multlined}[t]\dfrac{1}{2}\left[ c^2_g\Vert\nabla(\psi +\tau v)\Vert^{2}_{L^2} +\Vert v+\tau w\Vert^{2}_{L^2}+(\tau-\varepsilon/\zeta)\Vert \nabla\eta\Vert^{2}_{L^2, -g'}\right.  \\
	\left. +\varepsilon\tau (b- \tau c^2_g)/(1+\varepsilon) \Vert \nabla v\Vert^{2}_{L^2}
 \vphantom{\frac12}\right]. \end{multlined}
	\end{aligned}
	\end{equation}
Consequently, the left-hand side inequality in \eqref{equiv E and E0} holds. The right-hand side inequality follows analogously. 
\end{proof}
\noindent The next step is to derive a lower-order energy estimate for $E_1$. 
\begin{proposition} \label{Prop:E1} 
 Let $(\psi, v, w, \eta)$ be a smooth solution of the system \eqref{Main_System} with initial data \eqref{Main_System_IC}. Then the following estimate holds:
	\begin{equation} \label{Energy_Indentity}
	\frac{\textup{d}}{\dt}E_1(t)+( b -\tau c^2 )\Vert \nabla v(t)\Vert_{L^2}^{2}+\frac{%
		1}{2}\Vert \nabla \eta\Vert _{L^2, -g'}^{2}\leq \,
	|R^{(1)}(v + \tau w)| 
	\end{equation}
	for all $t\geq 0$, where the functional $R^{(1)}$ is defined in \eqref{Def_R}.
\end{proposition}

\begin{proof}
Looking at the definition \eqref{energy} of the energy $E_1$, we begin by obtaining an expression for $\tfrac12 c^2_g \tfrac{\textup{d}}{\textup{d}t}\|\nabla (\psi+\tau v)\|_{L^2}$.  It is clear that
\begin{equation} \label{eq1_}
\left(\psi+\tau v\right) _{t}=v+\tau w.
\label{eq1}
\end{equation}%
Multiplying the above equation by $-c^2_g \Delta \left( \psi+\tau v\right) $ and
integrating over $\mathbb{R}^{n}$ gives the identity
\begin{equation}  \label{eq2}
\begin{aligned}
&\dfrac{c^2_g }{2}\dfrac{\textup{d}}{\dt}\int_{\mathbb{R}^{n}}|\nabla (\psi+\tau
v)|^{2}\dx \\
=&\,\begin{multlined}[t] \tau c^2_g |\nabla v|^{2}+c^2_g \int_{\mathbb{R}%
	^{N}}\nabla \psi \cdot\nabla v \dx   
+\tau ^{2}c^2_g \int_{\mathbb{R}^{n}}\nabla v \cdot \nabla w \dx + \tau c^2_g
\int_{\mathbb{R}^{n}}\nabla w \cdot \nabla \psi \dx. \end{multlined}
\end{aligned}
\end{equation}%
To tackle the time derivative of the second term in the energy \eqref{energy}, we then multiply the second equation in the system \eqref{Main_System} by $-\tau
( b -\tau c^2_g )\Delta v$ and integrate over $\mathbb{R}^{n}$. By doing so, we obtain
\begin{equation} \label{eq_nabla_v}
\dfrac{1}{2}\tau ( b -\tau c^2_g )\dfrac{\textup{d}}{\dt}\int_{\mathbb{R}%
	^{n}}|\nabla v|^{2} \dx=\tau ( b -\tau c^2_g )\int_{\mathbb{R}%
	^{n}}\nabla w \cdot \nabla v \dx.  \label{eq3}
\end{equation}%
To handle the term $\tfrac{1}{2}\tfrac{\textup{d}}{\dt} \|v+\tau w\|_{L^2}^{2}$, we add the second equation in the system \eqref{Main_System} to the third one. Then the $w$ terms cancel out and we have 
\begin{align} \label{eq_1}
(v+\tau w)_{t}=b \Delta v+c^2_g \Delta \psi+\displaystyle%
\int_{0}^{\infty}g(s)\Delta \eta(s)\ds+2k vw. 
\end{align}%
Multiplying the above equation by $v+\tau w$ and integrating over $%
\mathbb{R}^{n}$ yields 
\begin{equation} 
\begin{aligned}\label{eq2}
&\frac{1}{2}\frac{\textup{d}}{\dt} \|v+\tau w\|_{L^2}^{2}\\
=& \, \begin{multlined}[t]-b
\int_{\mathbb{R}^{n}}|\nabla v|^{2}\dx- \tau b \int_{\mathbb{R}%
	^{n}}\nabla v \cdot \nabla w \dx-c^2_g \int_{\mathbb{R}^{n}}\nabla \psi  \cdot \nabla
v \dx  \\
-c^2_g  \tau \int_{\mathbb{R}^{n}}\nabla \psi \cdot \nabla w \dx-\int_{%
	\mathbb{R}^{n}} \int_{0}^{\infty}g(s) \nabla \eta
(s) \cdot \nabla v \ds\dx   \\
-\tau  \int_{\mathbb{R}^{n}} \int_{0}^{\infty 
}g(s) \nabla \eta(s) \cdot \nabla w
\ds \dx+  R^{(1)}(v + \tau w).   \end{multlined}
\end{aligned}
\end{equation}
We can further transform the first term on the right that contains the memory kernel by using the fact that $\eta _{t}+\eta _{s}=v$. Indeed, we have
\begin{align}
-\int_{\mathbb{R}^{n}}\int_{0}^{\infty}g(s)
\nabla \eta(s) \cdot \nabla v \ds \dx=& -\int_{%
	\mathbb{R}^{n}}\int_{0}^{\infty}g(s) \nabla \eta
(s) \cdot \nabla \eta _{t}(s) \ds\dx \\
& -\int_{\mathbb{R}^{n}} \int_{0}^{\infty}g(s)
\nabla \eta(s) \cdot \nabla \eta _{s}(s) \ds \dx.
\end{align}%
Integrating by parts with respect to $s$ in the second term on the right leads to
\begin{align}
-\int_{\mathbb{R}^{n}} \int_{0}^{\infty}g(s)
\nabla \eta(s)\cdot \nabla v \ds \dx=& -\dfrac{1}{2}\dfrac{%
	\textup{d}}{\dt}\int_{\mathbb{R}^{n}} \int_{0}^{\infty}g(s)|\nabla
\eta(s)|^{2} \ds\dx \\
& +\frac{1}{2}\int_{\mathbb{R}^{n}} \int_{0}^{\infty
}g^{\prime }(s)|\nabla \eta(s)|^{2} \ds \dx;
\end{align}
noting that the boundary terms vanish;
cf.~\cite{pata2009stability}.
Hence, we get 
\begin{align}
-\int_{\mathbb{R}^{n}}\int_{0}^{\infty}g(s)
\nabla \eta(s)\cdot \nabla v \ds \dx=& -\frac{1}{2}\dfrac{%
	\textup{d}}{\dt}\Vert \nabla \eta\Vert _{L^2,g}^{2} 
 -\frac{1}{2}\Vert \nabla \eta \Vert _{L^2, -g'}^{2}.
\end{align}%
Similarly, using the relation $\eta _{tt}+\eta _{ts}=w$ results in
\begin{align}
 &- \tau \int_{\mathbb{R}^{n}}\displaystyle\int_{0}^{\infty
}g(s) \nabla \eta(s) \cdot \nabla w \ds \dx \\
=& - \tau \int_{\mathbb{R}^{n}} \int_{0}^{\infty
}g(s) \nabla \eta(s) \cdot  \nabla \eta _{tt}(s)
\ds\dx - \tau \int_{\mathbb{R}^{n}} \int_{0}^{\infty
}g(s)\nabla \eta(s) \cdot  \nabla \eta _{ts}(s)
\ds \dx.
\end{align}%
Then, by integrating by parts with respect to $s$, we have 
\begin{equation}
\begin{aligned}
& - \tau \int_{\mathbb{R}^{n}} \int_{0}^{\infty
}g(s)\nabla \eta(s) \cdot  \nabla w \ds \dx \\
=&\, \begin{multlined}[t] - \tau \dfrac{\textup{d}}{\dt}\int_{\mathbb{R}^{n}} \int_{0}^{\infty
}g(s) \nabla \eta(s)\cdot  \nabla \eta _{t}(s)
\ds \dx 
+ \tau \int_{\mathbb{R}^{n}} \int_{0}^{\infty
}g(s)\nabla \eta _{t}(s) \cdot  \nabla \eta
_{t}(s) \ds \dx \\
 - \tau \dfrac{\textup{d}}{\dt}\int_{\mathbb{R}^{n}} \int_{0}^{\infty
}g(s)\nabla \eta(s) \cdot  \nabla \eta _{s}(s)\ds \dx  + \tau \int_{\mathbb{R}^{n}} \int_{0}^{\infty
}g(s) \nabla \eta _{t}(s)\cdot \nabla \eta
_{s}(s) \ds \dx. \end{multlined}
\end{aligned}
\end{equation}%
By inserting the derived identities into \eqref{eq2}, we infer
\begin{equation} \label{identity1}
\begin{aligned}
& \frac{1}{2}\dfrac{\textup{d}}{\dt}\left( \left\Vert v+\tau w\right\Vert _{L^{2}}^{2}+\|\nabla \eta\|^2_{L^2, g} +2 \tau \int_{\mathbb{R}^{n}}\displaystyle\int_{0}^{\infty
}g(s) \nabla \eta(s)\cdot \nabla v\ds
\dx\right)\\
=&\,\begin{multlined}[t]- b \left\Vert \nabla
v\right\Vert _{L^{2}}^{2}-\tau  b \int_{\mathbb{R}^{n}}\nabla
v \cdot \nabla w\dx-c^2_g
\int_{\mathbb{R}^{n}}\nabla \psi \cdot \nabla v\dx 
  \\ -c^2_g  \tau \int_{\mathbb{R}^{n}}\nabla \psi \cdot \nabla w\dx
-\frac{1}{2}\Vert \nabla \eta(s)\Vert _{L^2, -g'%
}^{2}+ \tau \int_{\mathbb{R}^{n}} \displaystyle\int_{0}^{t
}g(s) \nabla \eta _{t}(s) \cdot \nabla v\ds \dx\\+R^{(0)}(v + \tau w). \end{multlined}
\end{aligned}%
\end{equation}
By adding also equation \eqref{eq_nabla_v} to the above expression, we infer 
\begin{equation}
\begin{aligned}
\frac{1}{2}\dfrac{\textup{d}}{\dt}\left(E_1(t)-\tau \Vert \nabla \eta\Vert _{L^2, -g'%
}^{2}\right)
=&\, \begin{multlined}[t]-(b -\tau c^2_g )\left\Vert
\nabla v\right\Vert _{L^{2}}^{2}-\frac{1}{2}\Vert \nabla \eta 
(s)\Vert _{L^2, -g'}^{2}   \\
 +\tau \int_{\mathbb{R}^{n}}\displaystyle\int_{0}^{\infty
}g(s) \nabla \eta _{t}(s)\cdot \nabla v \ds\dx+2k(vw, v + \tau w)_{L^2}.\end{multlined}
\end{aligned}%
\end{equation}
To further transform the memory term on the right, we can substitute $\eta _{t}=v-\eta _{s}$. This action leads to
\begin{equation}
\begin{aligned}
\frac{1}{2}\dfrac{\textup{d}}{\dt}\left(E_1(t)-\tau \Vert \nabla \eta\Vert _{L^2,-g'%
}^{2}\right) =&\, \begin{multlined}[t]-( b -\tau c^2_g )\left\Vert
\nabla v\right\Vert _{L^{2}}^{2}-\frac{1}{2}\Vert \nabla \eta
(s)\Vert _{L^2, -g'}^{2} \\
+\tau \int_{\mathbb{R}^{n}} \int_{0}^{\infty}g(s)\nabla (v- \eta _{s}(s))\cdot \nabla v \ds\dx +R^{(1)}(v + \tau w). \end{multlined}
\end{aligned}
\end{equation}%
Integrating once by parts with respect to $s$ in the memory term yields
\begin{equation}
\begin{aligned}
\frac{1}{2}\dfrac{\textup{d}}{\dt}\left(E_1(t)-\tau \Vert \nabla \eta\Vert _{L^2,-g'%
}^{2}\right)
=& \, \begin{multlined}[t]-(b -\tau c^2_g )\left\Vert
\nabla v\right\Vert _{L^{2}}^{2}-\frac{1}{2}\Vert \nabla \eta
(s)\Vert _{L^2, -g'}^{2}+\tau (c^2 -c^2_g )\left\Vert \nabla
v\right\Vert _{L^{2}}^{2} \\
 +\tau \int_{\mathbb{R}^{n}} \int_{0}^{\infty}g^{\prime
}(s) \nabla \eta(s)\cdot \nabla v \ds
\dx+R^{(1)}(v + \tau w). \end{multlined}
\end{aligned}
\end{equation}%
We then again use the same trick of substituting $v=\eta _{t}+\eta _{s}$, which results in 
\begin{equation}
\begin{aligned}
\frac{1}{2}\dfrac{\textup{d}}{\dt}\left(E_1(t)-\tau \Vert \nabla \eta\Vert _{L^2, -g'%
}^{2}\right)
=& \, \begin{multlined}[t]-( b -\tau c^2 )\left\Vert
\nabla v\right\Vert _{L^{2}}^{2}-\frac{1}{2}\Vert \nabla \eta
(s)\Vert _{L^2, -g'}^{2}\\+\tau \int_{\mathbb{R}^{n}}
\int_{0}^{\infty}g^{\prime }(s) \nabla \eta(s) \cdot \nabla \eta
_{t} \ds\dx \\
 +\tau \int_{\mathbb{R}^{n}} \int_{0}^{\infty}g^{\prime
}(s) \nabla \eta(s)\cdot \nabla \eta _{s}
\ds \dx+R^{(1)}(v + \tau w).\end{multlined}
\end{aligned}
\end{equation}%
Finally, integrating by parts once again with respect to $s$ in the second memory term on the right leads to
\begin{equation}
\begin{aligned}
\frac{1}{2}\frac{\textup{d}}{\dt}E_1(t)
 = \, \begin{multlined}[t]-( b -\tau c^2 )\left\Vert
\nabla v\right\Vert _{L^{2}}^{2}-\frac{1}{2}\Vert \nabla \eta
\Vert _{L^2, -g'}^{2}-\frac{\tau }{2}\Vert \nabla \eta
\Vert _{L^2, g''}^{2}
+R^{(1)}(v + \tau w),\end{multlined}
\end{aligned}
\end{equation}%
which immediately yields \eqref{Energy_Indentity}, as desired.
\end{proof}

\subsection{Higher-order estimates} Next we analogously define the energy of the second order at time $t \geq 0$ as
\begin{equation} \label{E_2}
\begin{aligned}
E_{2}(t)
 =&\, \begin{multlined}[t]\frac{1}{2}\left[\vphantom{\int_{0}^{%
 		\infty}} c^2_g\left\Vert \Delta
 ( \psi+\tau v)\right\Vert _{L^{2}}^{2}+\tau(b -\tau c^2_g )\left\Vert
\Delta v\right\Vert _{L^{2}}^{2}+\Vert \nabla (v+\tau w)\Vert
_{L^2}^2\right. \\ \left.+\tau \Vert \Delta \eta \Vert^2_{L^2, -g'}    
+\Vert \Delta \eta\Vert^2_{L^{2}, g} +2\tau \int_{\mathbb{R}^{n}}\int_{0}^{%
\infty}g(s)\Delta v\Delta \eta(s)\ds\dx\right]. \end{multlined}
\end{aligned}
\end{equation}%
Observe that the last term above has an undefined sign; nevertheless, the other terms in the energy functional can absorb it. In fact, the functional $E_2$ is equivalent to $\mathscr{E}_2=\mathscr{E}_2[\vecc{\Psi}]$, which we introduced in \eqref{E_2_Eqv}. 
\begin{lemma} \label{Lemma_EquivE_1}
Assume that $b\geq \tau c^2> \tau c^2_g$.  Then there exist positive constants $C_{1}$ and $C_{2}$, such that
\begin{equation}\label{equiv E and V}
	C_{1} \mathscr{E}_2(t)\leq E_2(t)\leq C_{2} \mathscr{E}_2(t), 
\end{equation} 
for all $t\geq0.$
\end{lemma}
\begin{proof}
The proof follows the same steps as proof of Lemma~\ref{Lemma_EquivE_0}. We omit the details here.
\end{proof}
\noindent We move onto the derivation of a higher-order energy estimate for $E_2$, analogous to the one of Proposition~\ref{Prop:E1}.
\begin{proposition} \label{Prop:E2} Let $(\psi, v, w, \eta)$ be a smooth solution of the system \eqref{Main_System} with initial data \eqref{Main_System_IC}. Then the following inequality holds: 
\begin{equation}  \label{dE_1_Dt}
\frac{\textup{d}}{\dt}E_2(t)+\left(b -\tau c^2 \right) \left\Vert
\Delta v\right\Vert _{L^{2}}^{2}+\frac{1}{2}\Vert \Delta \eta
\Vert _{L^2, -g'}^{2}\leq \, |R^{(2)}(v + \tau w)|,  
\end{equation}%
for all $t\geq 0$, where the functional $R^{(2)}$ is defined in \eqref{Def_R}.
\end{proposition}
\begin{proof}
The proof follows by testing our problem with suitable test functions. Looking at the definition \eqref{E_2} of the higher-order energy, we first need to tackle the time derivative of the term $\tfrac12 c^2_g\|\Delta (\psi+\tau v)(t)\|^2_{L^2}$. Clearly,
\begin{equation}
\Delta (\psi+\tau v)_{t}=\Delta (v+\tau w).
\label{Equation_u_v_Lap}
\end{equation}%
Multiplying the above equation by $\Delta \left(\psi+\tau
v\right) $ and integrating over $\mathbb{R}^{n}$ results in
\begin{equation} \label{E_2_Second_Term}
\begin{aligned}
\frac{1}{2}\frac{\textup{d}}{\dt}\int_{\mathbb{R}^{n}}|\Delta (\psi+\tau
v)|^{2}\dx   
=&\, \begin{multlined}[t]\tau \int_{\mathbb{R}^{n}}\Delta w\Delta \psi\dx+\tau ^{2}\int_{%
	\mathbb{R}^{n}}\Delta w\Delta v\dx\\
+\int_{\mathbb{R}^{n}}\Delta
v\Delta \psi\dx+ \tau \int_{\mathbb{R}^{n}}|\Delta v|^{2}\dx. \end{multlined}
\end{aligned}
\end{equation}%
Next we work with the time derivative of $\|\Delta v(t)\|^2_{L^2}$. By applying the Laplacian to the second equation of the system \eqref{Main_System}, multiplying the resulting expression by $-\tau (b -\tau c^2_g )\Delta v$,
integrating over $\mathbb{R}^{n}$, and using integration by parts, we find 
\begin{equation} \label{v_Estimate_E_2}
\frac{1}{2}\tau (b -\tau c^2_g )\frac{\textup{d}}{\dt}\int_{\mathbb{R}%
^{n}}|\Delta v|^{2}\dx=\tau (b -\tau c^2_g )\int_{\mathbb{R}%
^{n}}\Delta w \, \Delta v \dx.  
\end{equation}%
To handle the time derivative of the third term in \eqref{E_2},  we apply the operator  $\Delta $ to \eqref{eq_1}
we get (in the sense of distribution)
\begin{align} 
(\Delta( v+\tau w))_{t}=b \Delta^2 v+c^2_g \Delta^2 \psi+\displaystyle%
\int_{0}^{\infty}g(s)\Delta^2 \eta(s)\ds+2k \Delta(vw). 
\end{align}
We multiply the above equation  by $ -(v+\tau w)$ and integrate over $\mathbb{R}^{n}$, yielding
\begin{equation} \label{E_2_First_Term}
\begin{aligned}
&\frac{1}{2}\frac{\textup{d}}{\dt}\int_{\mathbb{R}^{n}}|\nabla ( v+\tau
w)|^{2}\dx+b \int_{\mathbb{R}^{n}}\left\vert \Delta v\right\vert
^{2}\dx   \\
=&\, \begin{multlined}[t]-b \tau \int_{\mathbb{R}^{n}}\Delta v\Delta w\dx-c^2_g \int_{\mathbb{R}%
	^{n}}\Delta \psi\Delta (v+\tau w)\dx   \\
-\int_{\mathbb{R}^{n}}\int_{0}^{\infty}g(s)\Delta (v+\tau w)\Delta
\eta(s)ds\dx +R^{(2)}(v + \tau w). \end{multlined}
\end{aligned}
\end{equation}%
By summing up \eqref{E_2_First_Term}$+$\eqref{v_Estimate_E_2} $%
+c^2_g$\eqref{E_2_Second_Term}, we obtain 
\begin{equation} \label{E_2_1}
\begin{aligned}
&\begin{multlined}[t]\frac{1}{2}\frac{\textup{d}}{\dt}\left[ \left\Vert \nabla ( v+\tau
w)\right\Vert _{L^{2}}^{2}+\tau (b -\tau c^2_g
)\left\Vert \Delta v\right\Vert _{L^{2}}^{2}+c^2_g \left\Vert
\Delta ( \psi+\tau v)\right\Vert _{L^{2}}^{2}\right] \\
+\left( b
-\tau c^2_g \right) \left\Vert \Delta v\right\Vert _{L^{2}}^{2} \end{multlined}  \\
=&-\int_{\mathbb{R}^{n}}\int_{0}^{\infty}g(s)\Delta ( v+\tau
w)\Delta \eta(s)\ds\dx+R(\Delta( v+\tau
w)) \\
=&\,\begin{multlined}[t]- \int_{\mathbb{R}^{n}}\int_{0}^{\infty}g(s)\Delta v\Delta \eta
(s)\ds\dx-\tau \int_{\mathbb{R}^{n}}\int_{0}^{\infty}g(s)\Delta w\Delta
\eta(s)\ds\dx\\
+R^{(2)}(v + \tau w). \end{multlined}
\end{aligned}
\end{equation}
We next want to further transform the first two terms on the right-hand side. By taking the Laplacian of the last equation in \eqref{Main_System}, we obtain $\Delta v=\Delta \eta _{t}+\Delta \eta _{s}$. We can then use this relation to find that
\begin{equation}
\begin{aligned}
\int_{\mathbb{R}^{n}}\int_{0}^{\infty}g(s)\Delta v\Delta \eta
(s)\ds\dx =&\, \begin{multlined}[t]\frac{1 }{2}\frac{\textup{d}}{\dt}\int_{\mathbb{R}%
^{n}}\int_{0}^{\infty}g(s)|\Delta \eta(s)|^{2}\ds\dx \\
+ \int_{\mathbb{R}^{n}}\int_{0}^{\infty}g(s)\Delta \eta
_{s}\, \Delta \eta(s)\ds\dx. \end{multlined}
\end{aligned}
\end{equation}%
By integrating by parts with respect to $s$ in the last term, we infer
\begin{equation} \label{First_Term_Integral}
\begin{aligned}
\int_{\mathbb{R}^{n}}\int_{0}^{\infty}g(s)\Delta v\Delta \eta
(s)\ds\dx 
=&\,\frac{1}{2}\frac{\textup{d}}{\dt}\Vert \Delta \eta
\Vert _{L^2, g}^{2} +\frac{1 }{2}\int_{\mathbb{R}^{n}}\int_{0}^{\infty}g^{\prime
}(s)|\Delta \eta(s)|^{2}\ds\dx   \\[1mm]
=&\,\frac{1 }{2}\frac{\textup{d}}{\dt}\Vert \Delta \eta
\Vert _{L^2, g}^{2}-\frac{1}{2}\Vert \Delta \eta
\Vert _{L^2, -g'}^{2}.  
\end{aligned}
\end{equation}%
To tackle the second memory term on the right in equation \eqref{E_2_1}, we can use the relation $\Delta w=\Delta \eta _{tt}+\Delta \eta _{ts}$, which holds in the sense of distribution.  Doing so yields
\begin{equation} 
\begin{aligned}
&\tau \int_{\mathbb{R}^{n}}\int_{0}^{\infty}g(s)\Delta w\Delta \eta  
(s)\ds\dx   \\
=&\, \begin{multlined}[t]\tau \frac{\textup{d}}{\dt}\int_{\mathbb{R}^{n}}\int_{0}^{\infty}g(s)\Delta \eta
_{t}\Delta \eta(s)\ds\dx-\tau \int_{\mathbb{R}^{n}}\int_{0}^{\infty
}g(s)\Delta \eta _{t}(s)\Delta \eta _{t}(s)\ds\dx   \\
+\tau \frac{\textup{d}}{\dt}\int_{\mathbb{R}^{n}}\int_{0}^{\infty}g(s)\Delta \eta
_{s}\Delta \eta(s)\ds\dx-\tau \int_{\mathbb{R}^{n}}\int_{0}^{\infty
}g(s)\Delta \eta _{s}\Delta \eta _{t}(s)\ds\dx.  \end{multlined} 
\end{aligned}
\end{equation}
Since $v= \eta_t+\eta_s$, we further have
\begin{equation} \label{Second_Term_Integra}
\begin{aligned}
&\tau \int_{\mathbb{R}^{n}}\int_{0}^{\infty}g(s)\Delta w\Delta \eta  
(s)\ds\dx   \\
=&\, \tau \frac{\textup{d}}{\dt}\int_{\mathbb{R}^{n}}\int_{0}^{\infty}g(s)\Delta
v\Delta \eta(s)\ds\dx-\tau \int_{\mathbb{R}^{n}}\int_{0}^{\infty
}g(s)\Delta v\Delta \eta _{t}(s)\ds\dx.  
\end{aligned}
\end{equation}
Consequently, from \eqref{E_2_1}, \eqref{First_Term_Integral} and \eqref{Second_Term_Integra}, we have 
\begin{eqnarray*}
&&\frac{\textup{d}}{\dt}(E_2(t)-\frac{\tau}{2} \Vert \Delta \eta\Vert_{L^2, -g'}^2)+\left( b -\tau c^2_g \right) \left\Vert
\Delta v\right\Vert _{L^{2}}^{2}+\frac{1 }{2}\Vert \Delta \eta
\Vert _{L^2, -g'}^{2} \\
&=&\tau \int_{\mathbb{R}^{n}}\int_{0}^{\infty}g(s)\Delta v\Delta \eta
_{t}(s)\ds\dx+R^{(2)}(v + \tau w).  
\end{eqnarray*}
By using the fact that   
\begin{equation}
\begin{aligned}
\tau \int_{\mathbb{R}^{n}}\int_{0}^{\infty}g(s)\Delta v\Delta \eta
_{t}(s)\ds\dx =& \, \tau \int_{\mathbb{R}^{n}}\int_{0}^{\infty}g(s)\Delta
v(\Delta v-\Delta \eta _{s}(s))\ds\dx \\
=& \, \begin{multlined}[t] \tau \left( c^2 -c^2_g \right) \left\Vert \Delta v\right\Vert
_{L^{2}}^{2}
+\tau \int_{\mathbb{R}^{n}}\int_{0}^{\infty}g^{\prime}(s)\Delta v\Delta \eta(s)\ds\dx, \end{multlined}
\end{aligned}
\end{equation}%
we find that 
\begin{equation}\label{dE_1_dt_1}
\begin{aligned}
&\frac{\textup{d}}{\dt}(E_2(t)-\frac{\tau}{2} \Vert \Delta \eta\Vert_{L^2, -g'}^2)+\left( b -\tau c^2 \right) \left\Vert
\Delta v\right\Vert _{L^{2}}^{2}+\frac{1 }{2}\Vert \Delta \eta
\Vert _{L^2, -g'}^{2} \\
=&\, \tau \int_{\mathbb{R}^{n}}\int_{0}^{\infty}g^{\prime }(s)\Delta v\Delta
\eta(s)\ds\dx+R^{(2)}(v + \tau w).
\end{aligned}
\end{equation}%
The term on the right-hand side  of \eqref{dE_1_dt_1}  can be written as, by using the fact that $\Delta v=\Delta \eta+\Delta \eta_s$,  
\begin{equation}
\begin{aligned}
&\tau \int_{\mathbb{R}^{n}}\int_{0}^{\infty}g^{\prime }(s)\Delta v\Delta
\eta(s)\ds\dx\\
 =&\, \begin{multlined}[t]\tau \int_{\mathbb{R}^{n}}\int_{0}^{\infty}g^{\prime }(s)\Delta
\eta(s)\Delta\eta_t(s)\ds\dx 
+ \tau \int_{\mathbb{R}^{n}}\int_{0}^{\infty}g^{\prime }(s)\Delta
\eta(s)\Delta
\eta_s\ds\dx \end{multlined} \\
=&\, -\frac{\tau}{2}\frac{\textup{d}}{\dt}\Vert \Delta 
 \eta\Vert _{L^2, -g'}^{2}-\frac{\tau }{2}\Vert \Delta 
 \eta\Vert _{L^2, g''}^{2},
\end{aligned}
\end{equation}
where we integrated by parts with respect to $s$ in the second term. By plugging this identity into \eqref{dE_1_dt_1},  we deduce \eqref{dE_1_Dt}. This finishes the proof of Proposition~\ref{Prop:E2}. 
\end{proof} 
In order to capture the dissipation of the terms $\Vert \Delta( \psi+\tau v) \Vert_{L^2}$ and $\Vert \nabla(v+\tau w) \Vert_{L^2}$, we introduce two functionals $F_{1}$ and $F_2$ as
\begin{equation} \label{F_Functionals}
\begin{aligned} 
F_{1}(t)=\,\int_{\mathbb{R}^{n}}\nabla ( \psi+\tau v)\cdot \nabla (v+\tau w)\dx, \qquad
F_{2}(t)=\,-\tau \int_{\mathbb{R}^{n}} \nabla v  \cdot  \nabla (v+\tau w) \dx, 
\end{aligned}
\end{equation}
everywhere in time; see also~\cite{Racke_Said_2019}. We prove their properties in the following two lemmas. 
\begin{lemma}
\label{Lemma_F_1} Let $(\psi, v, w, \eta)$ be a smooth solution of the system \eqref{Main_System} with initial data \eqref{Main_System_IC}. For any $\epsilon _{0},\epsilon _{1}>0,$ it holds
\begin{equation} \label{F_1_Estimate}
\begin{aligned}
&\frac{\textup{d}}{\dt}F_{1}(t)+(c^2_g -\epsilon _{0}-(c^2 -c^2_g )\epsilon
_{1})\Vert \Delta (\psi+\tau v)\Vert _{L^{2}}^{2} \\
\leq&\, \begin{multlined}[t] \Vert \nabla (v+\tau w)\Vert _{L^{2}}^{2}+C(\epsilon
_{0})\Vert \Delta v\Vert _{L^{2}}^{2}
+C(\epsilon _{1})\Vert \Delta \eta\Vert _{L^{2},g}^{2}+|R^{(2)}(\psi + \tau v)|.\end{multlined}
\end{aligned}
\end{equation}%
\end{lemma}
\begin{proof}
	We first compute the derivative of the functional $F_1$ as
	\begin{equation} \label{dF_1}
	\begin{aligned}
	\frac{\textup{d}}{\dt}F_{1}(t)
	=&\,\begin{multlined}[t] -\int_{\mathbb{R}^{n}}\Delta ( \psi+\tau v) (v+\tau w)_t\dx-\int_{\mathbb{R}^{n}} (\psi+\tau v)_t \Delta (v+\tau w)\dx.\end{multlined}
	\end{aligned}
	\end{equation}
We clearly have to further transform the two terms on the right-hand side. Recall that 
\begin{equation} 
	(v+\tau w)_{t}=b \Delta v+c^2_g \Delta \psi+\displaystyle%
	\int_{0}^{\infty}g(s)\Delta \eta(s)\ds+2k vw.
\end{equation}%
Multiplying this equation by $-\Delta \left( 
\psi+\tau v\right) $ and integrating over $\R^n$ leads to
\begin{equation} \label{Lemma_identity1}
\begin{aligned}
&-\int_{\mathbb{R}^{n}} \Delta (\psi+\tau v)\, (v+\tau w)_{t} \dx \\
=&\, \begin{multlined}[t]- \int_{\mathbb{R}^{n}}(c^2_g \Delta \psi+b \Delta v)( \Delta
\psi+\tau \Delta v)\dx \\
- \int_{\mathbb{R}^{n}}\int_{0}^{\infty}g(s)\,\Delta \eta
(s)\,( \Delta \psi+\tau \Delta v)\ds\dx 
+R^{(2)}(\psi + \tau v). \end{multlined} 
\end{aligned}
\end{equation}%
We can conveniently rearrange the first term on the right as
\begin{equation}
\begin{aligned}
&-\int_{\mathbb{R}^{n}}(c^2_g \Delta \psi+b \Delta v)( \Delta
\psi+\tau \Delta v)\dx \\
=&\,-c^2_g \|\Delta ( \psi+\tau v)\|^2_{L^2}+(b -\tau c^2_g
)\int_{\mathbb{R}^{n}}\Delta v \Delta ( \psi+\tau  v)\dx .
\end{aligned}
\end{equation}
The second term on the right in \eqref{dF_1} can be written as
\begin{equation} 
\begin{aligned}
- \int_{\mathbb{R}^{n}}( \psi+\tau v)_{t}\Delta ( v+\tau
w)\dx=\,- \int_{\mathbb{R}^n}(v+\tau w)\Delta (v+\tau w)\dx 
=\,  \|\nabla (v+\tau w)\|^2_{L^2} .
\end{aligned}
\end{equation}%
By adding together \eqref{Lemma_identity1} and the above identity, and then integrating by parts in space, we obtain 
\begin{equation}\label{F_1_terms}
\begin{aligned}
& \begin{multlined}[t] \frac{\textup{d}}{\dt}F_{1}(t)+c^2_g \int_{\mathbb{R}^{n}}|\Delta ( \psi+\tau
v)|^{2}\dx \end{multlined}\\
=&\,\begin{multlined}[t] \int_{\mathbb{R}^{n}}|\nabla (v+\tau w)|^{2}\dx-(b-\tau c^2_g)\int_{\mathbb{R}^{n}}\Delta v \, (\Delta
\psi+\tau \Delta v)\dx\\+  R^{(2)}(\psi + \tau v)
-\int_{\mathbb{R}^{n}}\int_{0}^{\infty}g(s)\, \Delta \eta
(s)\,\Delta( \psi+\tau v)\ds\dx .\end{multlined}
\end{aligned}
\end{equation}%
Applying Young's inequality results in \eqref{F_1_Estimate} for any $\epsilon _{0},\epsilon _{1}>0$. 
\end{proof}
  \noindent We next prove an important energy property of the functional $F_2$.
     \begin{lemma}\label{Lemma_F_2}
     Let $(\psi, v, w, \eta)$ be a smooth solution of the system \eqref{Main_System} with initial data \eqref{Main_System_IC}. For any $\epsilon_{2},\epsilon_{3}>0,$  we have 
     	\begin{equation}\label{F_2_Estimate}
     	\begin{aligned}
     	&\frac{\textup{d}}{\dt}F_{2}(t)+(1-\epsilon_{3})\Vert\nabla(v+\tau w)\Vert^{2}_{L^2} \\
	\leq&\, \begin{multlined}[t]\epsilon_{2}\Vert\Delta(\psi +\tau v)\Vert^{2}_{L^2}  +C(\epsilon_{3},\epsilon_{2})(\Vert \Delta v\Vert^{2}_{L^2}+\Vert \nabla v\Vert^{2}_{L^2}) 
     	+\frac{1}{2}\Vert \nabla\eta\Vert^{2}_{L^2, g}+|R^{(2)}(\tau v)|,\end{multlined}
     	\end{aligned}
     	\end{equation}
     	where the functional $R^{(2)}$ is defined in \eqref{Def_R}.
\end{lemma}         
\begin{proof}
	We can express the derivative of the functional $F_2$ as
	\begin{equation} \label{dF_2}
	\begin{aligned}
	\frac{\textup{d}}{\dt}F_{2}(t)=&\, \begin{multlined}[t] \tau \int_{\mathbb{R}^{n}} v_t \, \Delta (v+\tau w)\dx+ \tau \int_{\mathbb{R}^{n}} \Delta v \, (v+\tau w)_{t}\dx \end{multlined}\\
=&\, 
\tau \int_{\mathbb{R}^{n}}w\Delta (v+\tau w)\dx+ \tau \int_{\mathbb{R}^{n}} \Delta v \, (v+\tau w)_{t}\dx ,
\end{aligned}
\end{equation}%
where the second line follows from $v_t=w$.  To further transform the second term on the right, we multiply equation \eqref{eq_1} by $ \tau \Delta v$. This action leads to
\begin{equation}
\begin{aligned}
& \tau \int_{\mathbb{R}^{n}}(v+\tau w)_{t}\, \Delta v\dx \\
=&\, \begin{multlined}[t]\int_{\mathbb{R}^{n}}\left(\tau c^2_g \Delta( \psi+\tau v) +\tau(b-\tau c^2_g)
\Delta v+(
v+\tau w)\right. \\ \left.-(v+\tau w)\right)\Delta v\dx 
+ \tau\int_{\mathbb{R}^{n}}\int_{0}^{\infty}g(s)\Delta \eta(s)\Delta
v\ds\dx+R^{(2)}(\tau v).\end{multlined}
\end{aligned}
\end{equation}%
By plugging this identity into \eqref{dF_2}, we obtain 
\begin{equation}
\begin{aligned}
&\frac{\textup{d}}{\dt}F_{2}(t)+ \int_{\mathbb{R}^{n}}|\nabla ( v+\tau
w)|^{2}\dx\\
=&\, \begin{multlined}[t]\tau c^2_g \int_{\mathbb{R}^{n}}\Delta (\psi+\tau v)\Delta v\dx+\tau (b -\tau c^2_g )\int_{\R^{}}|\Delta
v|^{2}\dx \\
+\int_{%
	\mathbb{R}^{n}}\nabla (v+\tau w)\cdot \nabla v\dx + \tau \int_{\mathbb{R%
}^{n}}\int_{0}^{\infty}g(s)\Delta \eta(s)\Delta v\ds\dx+ R^{(2)}(\tau v). \end{multlined}
\end{aligned}
\end{equation}%
By additionally applying Young's inequality with $\epsilon _{2},\epsilon _{3}>0$, we arrive at the final estimate \eqref{F_2_Estimate}.
\end{proof}
\subsection{The Lyapunov functional} We are now ready to introduce the Lyapunov functional $\mathcal{L}$ as 
\begin{eqnarray}\label{Lyapunov}
\mathcal{L}(t)=L_1 (E_1(t)+E_2(t)+\varepsilon\tau \Vert w\Vert^2 _{L^{2}})+F_1(t)+L_2F_2(t),     
\end{eqnarray}
for $t \geq 0$. The positive constants $L_{1}$ and $L_{2}$ should be sufficiently large and the constant $\varepsilon>0$ small enough; we will make them more precise below. \\
\indent This Lyapunov functional can be made equivalent to $\mathscr{E}_1+\mathscr{E}_2+\Vert w\Vert_{L^2}^2$, where the energies $\mathscr{E}_1$ and $\mathscr{E}_2$ are defined in \eqref{E_1_Eqv} and \eqref{E_2_Eqv}, respectively. We prove this statement next. 
\begin{lemma}
\label{Lemma_Equivl} Let $b\geq \tau c^2>\tau c^2_g$. There exist positive constants $C_{1}$ and $C_{2}$, 
such that 
\begin{equation}  \label{Eq_L_E}
C_{1}(\mathscr{E}_1(t)+\mathscr{E}_2(t)+\Vert w\Vert_{L^2}^2)\leq \mathcal{L}(t)\leq C_{2}(\mathscr{E}_1(t)+\mathscr{E}_2(t)+\Vert
w\Vert_{L^2}^2), 
\end{equation}
for all $t\geq 0$, provided that the constant $L_1$ in the Lyapunov functional \eqref{Lyapunov} is chosen large enough.
\end{lemma}
\begin{proof}
In view of \eqref{Lyapunov}, we are missing the bounds on $F_1$ and $F_2$ to arrive at our claim. We can estimate these terms in the Lyapunov functional as follows:
\begin{equation} \label{est_1}
\begin{aligned}
|F_1(t)| \leq \,  \Vert \nabla( \psi+\tau v)(t)\Vert _{L^2}\Vert \nabla( v +\tau w)(t) \Vert_{L^2} 
\lesssim\, E_1(t)+E_2(t),
\end{aligned}
\end{equation}
and
\begin{equation} \label{est_2}
\begin{aligned}
|F_2(t)| \leq \,  \tau \Vert \nabla v(t) \Vert _{L^2}\Vert \nabla( v +\tau w)(t) \Vert_{L^2} 
\lesssim E_1(t)+E_2(t)
\end{aligned}
\end{equation}
for all $t \geq 0$. Hence, there exists $C^\star=C^\star(\tau, c^2_g, b, L_2)>0$ such that
\begin{equation}
\begin{aligned}
|\mathcal{L}(t)-L_1 (E_1(t)+E_2(t)+\varepsilon\tau \Vert w\Vert^2 _{L^{2}})|
\leq\, C^\star (E_1(t)+E_2(t)+\varepsilon \tau \Vert w\Vert^2 _{L^{2}}).
\end{aligned}
\end{equation}
Choosing $L_1$ large enough so that
\begin{equation} \label{L1_large}
L_1>C^\star=C^\star(\tau, c^2_g, b, L_2)
\end{equation}
 leads to the estimates given in \eqref{Eq_L_E}.
\end{proof}
\noindent We next derive an energy bound for the Lyapunov functional.
\begin{proposition} Let $b>\tau c^2>\tau c^2_g$. There exist a constant $L_1>0$ large enough and a constant $\varepsilon>0$ small enough such that the Lyapunov functional, defined in \eqref{Lyapunov},  satisfies
	\begin{equation}\label{Lyap_Main}
	\begin{aligned}
	&\begin{multlined}[t]\frac{\textup{d}}{\dt}\mathcal{L}(t)+ \Vert \nabla
	v(t)\Vert _{L^2}^{2}+\Vert \nabla \eta \Vert _{L^2, -g'}^{2}+\mathscr{E}_2[\Psi](t)
	+\Vert w(t)\Vert
	_{L^{2}}^{2} \end{multlined}\\[1mm]
	\lesssim&\,\begin{multlined}[t]   |R^{(1)}(v+\tau w)|+|R^{(2)}(v+\tau w)|
	+|R^{(1)}(w)|
	+|R^{(2)}(\psi+\tau v)| +|R^{(2)}(\tau v)|,\end{multlined}
	\end{aligned}
	\end{equation}
	for all $t\in [0,T]$, where the functionals $R^{(1)}$ and $R^{(2)}$ are defined in \eqref{Def_R}, and the energy $\mathscr{E}_2$ in \eqref{E_2_Eqv}.
\end{proposition}
\begin{proof}
To derive the desired estimate, we have to get a bound on $\tfrac{\textup{d}}{\dt}\Vert w\Vert
_{L^{2}}^{2}$ first. By multiplying the third equation in the system \eqref{Main_System} by $w$ and
	integrating over $\mathbb{R}^{n}$, we infer
	\begin{equation}\label{w_Energy}
	\begin{aligned}
	&\frac{1}{2}\frac{\textup{d}}{\dt}\int_{\mathbb{R}^{n}}\tau \left\vert w\right\vert
	^{2}\dx+\int_{\mathbb{R}^{n}} |w|^2 \dx \\
	\leq&\, C(\left\Vert \Delta \psi\Vert _{L^{2}}+\Vert \Delta v\right\Vert
	_{L^{2}}+\Vert\Delta\eta\Vert_{L^2, g})\Vert w\Vert_{L^2}+|R^{(1)}(w)|.
	\end{aligned}
	\end{equation}
By applying Young's inequality to the first term on the right, we obtain
	\begin{equation} \label{E_0_Energy}
	\begin{aligned}
	&\frac{1}{2}\frac{\textup{d}}{\dt}\|w(t)\|^2_{L^2}+\frac{1}{2}\|w\|^2_{L^2}\\
	\leq&\,
	C(\Vert
	\Delta (\psi+\tau v)\Vert_{L^2}^2+\Vert \Delta v\Vert _{L^{2}}^2+\Vert\Delta\eta\Vert^{2}_{L^2, g})+|R^{(1)}(w)| .
	\end{aligned}
	\end{equation}
Collecting the derived bounds $ \eqref{Energy_Indentity}+\eqref{dE_1_Dt}+2\varepsilon \eqref{E_0_Energy}$, we get 
	\begin{equation}  \label{E_Wieghted_2}
	\begin{aligned}
	&\begin{multlined}[t]\frac{\textup{d}}{\dt}\left(E_1(t)+E_2(t)+\varepsilon\tau \Vert w\Vert ^2_{L^{2}}\right)+(b
	-\tau c^2 ) (\left\Vert \nabla v\right\Vert
	_{L^{2}}^{2}+\Vert \Delta v\Vert _{L^{2}}^{2})\\
	+\varepsilon
	 \Vert w\Vert _{L^{2}}^{2} +\frac{1}{2}\Vert \nabla \eta\Vert _{L^2,-g'
	}^{2}+\frac12 \Vert \Delta \eta\Vert _{L^2,-g'}^{2}  \end{multlined} \\
	\leq& \,\begin{multlined}[t] 2C\varepsilon(\Vert \Delta (\psi+\tau v)\Vert_{L^{2}}^{2}+\Vert \Delta
	v\Vert _{L^{2}}^{2})+\Vert\Delta\eta\Vert^{2}_{L^2, g}  \\[1mm]
	+|R^{(1)}(v+\tau w)|+|R^{(2)}(v+\tau w)|+2\varepsilon|R^{(1)}(w)|.  \end{multlined}
	\end{aligned}
	\end{equation}
Note that the first term on the left in the brackets is equal to $L_1^{-1}(\mathcal{L}(t)-F_1(t)-L_2F_2(t))$. Taking into account Lemmas \ref{Lemma_F_1} and \ref{Lemma_F_2}
as well as assumption (G3) on the memory kernel, we obtain 
\begin{equation}\label{L_dt_Functional}
\begin{aligned}
&\begin{multlined}[t]\frac{\textup{d}}{\dt}\mathcal{L}(t)+L_1\varepsilon \Vert w\Vert _{L^{2}}^{2}+\left[
L_1/2-\Lambda_0/\zeta-2CL_1\varepsilon\right]\left( \Vert \nabla \eta\Vert _{L^2, -g'}^{2}+\Vert \Delta
\eta\Vert _{L^2, -g'}^{2}\right) \\
+\left[ L_1\left( b -\tau
c^2 \right) -2L_1C\varepsilon-C(\epsilon _{0})-C(\epsilon
_{3},\epsilon _{2})L_2\right] (\left\Vert \nabla v\right\Vert
_{L^{2}}^{2}+\left\Vert \Delta v\right\Vert _{L^{2}}^{2}) \\[2mm]
+\left[ c^2_g -\epsilon _{0}-(c^2 -c^2_g )\epsilon _{1}-2C\varepsilon
L_1-\epsilon _{2}L_2\right] \Vert \Delta ( \psi+\tau v)\Vert
_{L^{2}}^{2}\\
+\left[ L_2(1 -\epsilon _{3})-1 \right] \Vert \nabla (
v+\tau w)\Vert _{L^{2}}^{2}
\end{multlined} \\
\leq& \, \begin{multlined}[t] \Lambda_1 \left\{ |R^{(1)}(v+\tau w)|+|R^{(2)}(v+\tau w)|
+|R^{(1)}(w)|
\right.\\ \left.+|R^{(2)}(\psi+\tau v)| +|R^{(2)}(\tau v)|\right\},\end{multlined}
\end{aligned}
\end{equation}
where $\Lambda_0$ and $\Lambda_1$ are generic positive constants that depend on $L_1$, $L_2$, $\epsilon_0,\dots$, yet $\Lambda_0$  is independent of $\varepsilon$.   \\
\indent In the above estimate, the constants $\epsilon_{0}$,  $\epsilon_{1}$, $\epsilon_{2}$, $L_1$, $L_2$, and $\varepsilon$ can be chosen in such a way that the coefficients on the left-hand side of \eqref{L_dt_Functional} are positive. This outcome can be achieved as follows: we pick $\epsilon_{3}>0$ small enough such that $\epsilon_{3}<1$. Then we can select $\epsilon_1=\epsilon_{0}>0$ and $\varepsilon>0$ small enough such that 
$$ \epsilon_{0}<\dfrac{c^2_g}{1+(c^2-c^2_g)},\qquad \text{and}\qquad \varepsilon<\frac{b-\tau c^2}{2C}.$$
Afterwards, we take $L_2$ large enough such that $$L_2>\frac{1}{1-\epsilon_{3}}.$$ 
Once $L_2$ and $\epsilon_{0}$ are fixed, we select  $\epsilon_{2}>0$ small enough such that
$$\epsilon_{2}<\dfrac{c^2_g-\epsilon_{0}(1+(c^2-c^2_g))}{L_2}.$$
Keeping in mind the assumption $b>\tau c^2,$ we take $L_1$ large enough such that condition \eqref{L1_large} holds together with
\begin{eqnarray}\label{N_0_1}
L_1\geq \max\left\lbrace \dfrac{C(\epsilon_{0})+L_2C(\epsilon_{2},\epsilon_{3})}{b-\tau c^2},\dfrac{2\Lambda_0}{\zeta}\right\rbrace.
\end{eqnarray} 
Finally, we decrease $\varepsilon>0$ additionally so that  
\begin{equation}
\varepsilon<\min \left( \frac{L_1\left( b -\tau c^2
	\right) -C(\epsilon _{0})-C(\epsilon _{3},\epsilon _{2})L_1}{2CL_1},%
\frac{L_1/2- \Lambda _{0}/\zeta}{2CL_1}\right) .
\end{equation} 
Consequently, we obtain the desired estimate  \eqref{Lyap_Main}.     
\end{proof}
Now, by integrating estimate \eqref{Lyap_Main} over the time interval $(0, \sigma)$ for $\sigma \in (0,t)$ and then taking the supremum over time, we obtain
\begin{equation} \label{energy_est_1}
\begin{aligned}
 |\vecc{\Psi}|^2_{\mathcal{E}(t)}+|\vecc{\Psi}|^2_{\mathcal{D}(t)}
\lesssim&\, \begin{multlined}[t] |\vecc{\Psi}|^2_{\mathcal{E}(0)}
+ \int_0^t \left\{ |R^{(1)}(v+\tau w)|+|R^{(2)}(v+\tau w)|
+|R^{(1)}(w)|\right.\\ \left.
+|R^{(2)}(\psi+\tau v)| +|R^{(2)}(\tau v)| \right\} \, \textup{d} \sigma, \end{multlined}
\end{aligned}
\end{equation}
where we have additionally exploited the equivalence of the Lyapunov functional and $\mathscr{E}_1+\mathscr{E}_2+\|w\|^2_{L^2}$ given in \eqref{Eq_L_E}.
\subsection{Estimates of the right-hand side terms}
To finalize the energy bound, we have to estimate the remaining $R^{(1)}$ and $R^{(2)}$ terms. We wish to bound each of these terms by $|\vecc{\Psi}|_{\mathcal{E}(t)}  |\vecc{\Psi}|^2_{\mathcal{D}(t)}$ multiplied by some positive constant $C$ that is independent of $t$. The estimates are split into two lemmas.
\begin{lemma}\label{Lemma:R_a}
Let $\vecc{\Psi}=(\psi, v, w, \eta)^T$ be a smooth solution of the system \eqref{Main_System} with initial data \eqref{Main_System_IC}. For all $t \in [0,T]$, it holds
\begin{equation}  \label{I_1_Estimate_N_2}
\begin{aligned}
&\int_0^t |R^{(1)}(v+\tau w) (\sigma)|\,\textup{d} \sigma+\int_0^t |R^{(1)}(w)(\sigma)|\, \textup{d}\sigma\lesssim |\vecc{\Psi}|_{\mathcal{E}(t)}  |\vecc{\Psi}|^2_{\mathcal{D}(t)},
\end{aligned}
\end{equation}
where the functional $R^{(1)}$ is defined in \eqref{Def_R} and the energy semi-norms $|\cdot|_{\mathcal{E}(t)}$ and $|\cdot|_{\mathcal{D}(t)}$ in \eqref{EnergyNorm} and \eqref{DissipativeEnergyNorm}, respectively.
\end{lemma}
\begin{proof}
By employing H\"older's inquality, we can proceed as follows:
\begin{equation}  \label{I_1_Main}
\begin{aligned}
|R^{(1)}(v+\tau w)|=& \, \left \vert 2k \int_{\mathbb{R}^n} vw(v+\tau w)\dx\right\vert\notag\\
\leq& \, 2  |k| \Vert w\Vert_{L^2}\Vert v\Vert_{L^4}^2+2 \tau |k|\Vert v\Vert_{L^2}\Vert
w\Vert_{L^4}^2.
\end{aligned}
\end{equation}
We can then rely on the Ladyzhenskaya interpolation inequality \eqref{Ladyz_Ineq}. We thus have for the first term on the right
\begin{equation}
\begin{aligned}
2  |k| \Vert w\Vert_{L^2}\Vert v\Vert_{L^4}^2\lesssim &\, \Vert w\Vert_{L^2}\Vert
v\Vert^{2(1-n/4)}_{L^2}\Vert \nabla v\Vert^{n/2}_{L^2} \\
\lesssim&\,  \|w\|_{L^2}(\|v\|^2_{L^2}+\|\nabla v\|^2_{L^2}),
\end{aligned}
\end{equation}
where we have also employed Young's inequality in the second line. Similarly, the second term can be estimated as
\begin{equation}\label{J_2_N_2}
\begin{aligned}
2 \tau |k| \Vert v\Vert_{L^2}\Vert w\Vert_{L^4}^2\lesssim& \, \Vert v\Vert_{L^2}\Vert
w\Vert^{2(1-n/4)}_{L^2}\Vert \nabla w\Vert^{n/2}_{L^2}\\
\lesssim&\, \|v\|_{L^2}(\|w\|^2_{L^2}+\|\nabla v\|^2_{L^2}).
\end{aligned}
\end{equation}
Altogether, this strategy yields
\begin{equation}
\begin{aligned}
\int_0^t |R^{(1)}(v+\tau w)(\sigma)|\, \textup{d} \sigma \lesssim & \, \begin{multlined}[t] \sup_{0\leq \sigma\leq t} \Vert w(\sigma)\Vert_{L^2}
\int_0^t (\Vert w(\sigma)\Vert_{L^2}^2+\Vert \nabla v(\sigma)\Vert_{L^2}^2) \, \textup{d}\sigma \\
+  \sup_{0\leq \sigma \leq t} \Vert v(\sigma)\Vert_{L^2}\int_0^t(\Vert
w(\sigma)\Vert_{L^2}^2 \, \textup{d}\sigma+\Vert \nabla w(\sigma)\Vert_{L^2}^2) \, \textup{d}\sigma.\end{multlined}
\end{aligned}
\end{equation}
 By additionally using the fact that
\begin{equation}\label{v,Nabla_w}
\begin{aligned}
&\Vert v(t)\Vert_{L^2}\leq \tau \Vert  w(t)\Vert_{L^2}+  \Vert
\nabla (v+\tau w)(t)\Vert_{L^2}, \\
&\Vert \nabla w(t)\Vert_{L^2}\leq \frac{1}{\tau}\Vert \nabla v(t)\Vert_{L^2}+ \frac{1}{\tau}\Vert
\nabla (v+\tau w)(t)\Vert_{L^2},
\end{aligned}
\end{equation}
for all $t$, we find that the first term on the left in \eqref{I_1_Estimate_N_2} can be bounded by $|\vecc{\Psi}|_{\mathcal{E}(t)}  |\vecc{\Psi}|^2_{\mathcal{D}(t)}$ up to a constant. The second term can be estimated  directly by noting that 
 \begin{equation} 
\begin{aligned}
|R^{(1)}(w)|\leq&\, 2|k|\Vert v\Vert_{L^4}\Vert w\Vert_{L^4}\Vert w\Vert_{L^2}\\
\lesssim& \,  \Vert v\Vert_{L^2}^{1-n/4}\Vert \nabla v\Vert_{L^2}^{n/4}\Vert w\Vert_{L^2}^{1-n/4}\Vert \nabla w\Vert_{L^2}^{n/4}\Vert w\Vert_{L^2}\\
\lesssim& \, (\Vert v\Vert_{L^2}+\Vert \nabla v\Vert_{L^2})(\Vert w\Vert_{L^2}^2+\Vert \nabla w\Vert_{L^2}^2)
\end{aligned}
\end{equation}
and recalling the above bounds on $\|v(t)\|_{L^2}$ and $\|\nabla w(t)\|_{L^2}$.
\end{proof}

\begin{lemma}\label{Lemma:R_b}
Let $\vecc{\Psi}=(\psi, v, w, \eta)^T$ be a smooth solution of the system \eqref{Main_System} with initial data \eqref{Main_System_IC}. Then it holds
\begin{equation}   \label{R_1_Estimate_N_2}
\begin{aligned}
\begin{multlined}[t]\int_0^t |R^{(2)}(v+\tau w)(\sigma)|\, \textup{d}\sigma+\int_0^t |R^{(2)}(\psi+\tau v)(\sigma)|\, \textup{d}\sigma\\
+\int_0^t |R^{(2)}( \tau  v)(\sigma)|\, \textup{d}\sigma
\lesssim |\vecc{\Psi}|_{\mathcal{E}(t)}|\vecc{\Psi}|^2_{\mathcal{D}(t)},\end{multlined}
\end{aligned}
\end{equation} 
for all $t \geq 0$, where the functional $R^{(2)}$ is defined in \eqref{Def_R} and the energy semi-norms $|\cdot|_{\mathcal{E}(t)}$ and $|\cdot|_{\mathcal{D}(t)}$ in \eqref{EnergyNorm} and \eqref{DissipativeEnergyNorm}, respectively.
\end{lemma}
\begin{proof}
We only estimate the first term on the left in \eqref{R_1_Estimate_N_2}, the second and third one can be bounded analogously.  By applying H\"older's inequality,  we obtain 
\begin{equation} \label{est_R_last}
\begin{aligned}
& |R^{(2)}( v +\tau w)|\\
\leq&\, 2|k|\Vert w\Vert_{L^4}\Vert \nabla v \Vert_{L^4}\Vert \nabla ( v +\tau w)\Vert _{L^2}+2|k| \Vert v \Vert_{L^\infty}\Vert \nabla w \Vert_{L^2}\Vert \nabla ( v +\tau w)\Vert_{L^2}
\end{aligned}
\end{equation}
for all times. For the first term on the right, we can then use the the Ladyzhenskaya interpolation inequality \eqref{Ladyz_Ineq} in two- and three-dimensions to obtain
\begin{equation}
\begin{aligned}
&2|k| \Vert w\Vert_{L^4}\Vert \nabla v \Vert_{L^4}\Vert \nabla ( v +\tau w)\Vert _{L^2}\\
 \lesssim& \,  \|w\|^{1-n/4}_{L^2}\|\nabla w\|^{n/4}_{L^2}\|\nabla v\|^{1-n/4}_{L^2}\|\nabla^2 v\Vert^{n/4}_{L^2}\Vert \nabla ( v +\tau w)\Vert _{L^2} .
 \end{aligned}
 \end{equation}
 From here, by employing Young's inequality and the bound \eqref{v,Nabla_w} for $\|\nabla w\|_{L^2}$, we have
\begin{equation}
\begin{aligned} 
&2|k| \Vert w\Vert_{L^4}\Vert \nabla v \Vert_{L^4}\Vert \nabla ( v +\tau w)\Vert _{L^2}\\ 
 \lesssim&\, (\|w\|_{L^2}+\|\nabla w\|_{L^2})(\Vert \nabla ( v +\tau w)\Vert^2 _{L^2}+\|\nabla v\|^2_{L^2}+\|\Delta v\|^2_{L^2})\\
 \lesssim&\, (\|w\|_{L^2}+\|\nabla v\|_{L^2}+\nabla ( v +\tau w)\Vert _{L^2})(\Vert \nabla ( v +\tau w)\Vert^2 _{L^2}+\|\nabla v\|^2_{L^2}+\|\Delta v\|^2_{L^2}).
\end{aligned}
\end{equation}
The second term on the right in \eqref{est_R_last} we can estimate as follows: 
\begin{equation}
\begin{aligned}
& 2|k|\Vert v \Vert_{L^\infty}\Vert \nabla w \Vert_{L^2}\Vert \nabla ( v +\tau w)\Vert_{L^2}\\
\lesssim& \,(\|v\|_{L^2}+\|\nabla v\|_{L^2}+\|\Delta v\|_{L^2})\Vert \nabla w \Vert_{L^2}\Vert \nabla ( v +\tau w)\Vert_{L^2} \\
\lesssim&\, (\|v+\tau w\|_{L^2}+\|w\|_{L^2}+\|\nabla v\|_{L^2}+\|\Delta v\|_{L^2})\Vert \nabla w \Vert_{L^2}\Vert \nabla ( v +\tau w)\Vert_{L^2} .
\end{aligned}
\end{equation}
Consequently, we can deduce that 
\begin{equation}\label{R_1_1_Est}
\begin{aligned}
\int_0^t|R^{(2)}( v +\tau w)(\sigma)| \, \textup{d} \sigma
\lesssim&\,\begin{multlined}[t]  \sup_{\sigma \in [0,t]}(\|v+\tau w\|_{H^1}+\|w\|_{L^2}+\|\nabla v\|_{L^2}+\|\Delta v\|_{L^2})\\
 \times \int_0^t(\Vert \nabla ( v +\tau w)\Vert^2 _{L^2}+\|\nabla v\|^2_{L^2}+\|\Delta v\|^2_{L^2}) \, \textup{d} \sigma, \end{multlined}
\end{aligned}
\end{equation}
from which the first estimate in \eqref{R_1_Estimate_N_2} follows.
\end{proof}
Altogether, our previous considerations allow us to conclude that if a smooth solutions of the system \eqref{Main_System} with initial data \eqref{Main_System_IC} exists on $[0,T]$, it must satisfy the estimate
\begin{eqnarray*}
	|\vecc{\Psi}|_{\mathcal{E}(t)}+|\vecc{\Psi}|_{\mathcal{D}(t)} \lesssim |\vecc{\Psi}|_{\mathcal{E}(0)}+ |\vecc{\Psi}|_{\mathcal{E}(t)}|\vecc{\Psi}|^2_{\mathcal{D}(t)}, \quad t \in [0,T]. 
\end{eqnarray*}
We next deal with the issue of existence of such a solution.
 \section{Local solvability of the JMGT equation with memory} \label{Sec:LocalExistence}
 In this section, we rely on the Banach fixed-point theorem to show the local well-posedness of our problem in $\R^n$, where $n \in \{2,3\}$. 
 \subsection{Linear local existence theory}
We begin by extending a linear existence result from~\cite{Bounadja_Said_2019} to allow for the possibility of having a source term. We recall how the Hilbert space $\mathcal{H}^1$ is defined in \eqref{Hilbert_spaces} and additionally introduce the domain of the operator $\mathcal{A}$ as
\begin{equation} \label{D(A)}
\begin{aligned}
D(\mathcal{A})
=\,  \left\{(\psi, v, w, \eta)^T \in \mathcal{H}^{1} \left\vert\rule{0cm}{1cm}\right. \begin{matrix}
w \in H^{1}(\R^n), \\[2mm]
\dfrac{c^2_g}{\tau} \Delta \psi+\dfrac{b}{\tau} \Delta v+\dfrac{1}{\tau}\displaystyle \int_0^\infty g(s)\Delta \eta(s)  \in L^2(\R^n), \\[4mm]
\eta \in D(\mathbb{T}) \end{matrix} \right\},
\end{aligned}
\end{equation}
where the operator $\mathbb{T}$ is defined in \eqref{def_T_eta}. We can now state a well-posedness result for a linearization of our problem.
 \begin{proposition} \label{Prop:LinExistence}
Assume that $b> \tau c^2 > \tau c^2_g$ and let the final time $T>0$ be given. Furthermore, assume that $(\psi_0, \psi_1, \psi_2) \in H^2(\R^n) \times H^2(\R^n) \times H^1(\R^n)$ and that a source term is given by \[F=[0, 0, f, 0]^T \in C^1([0,T]; \mathcal{H}^1) \cap C([0,T]; D(\mathcal{A})),\] where $n \in \{2, 3\}$. Then the initial-value problem 
 	\begin{equation} \label{IP_F}
 	\begin{cases}
 	\partial_t \vecc{\Psi}-\mathcal{A} \vecc{\Psi} = F, \\
 	\vecc{\Psi} \vert_{t=0}=\vecc{\Psi}_0
 	\end{cases}
 	\end{equation}
 	has a unique solution $\vecc{\Psi} \in  C^1([0,T]; \mathcal{H}^1) \cap C([0,T]; D(\mathcal{A}))$. This solution  satisfies the following energy estimate:
 	\begin{equation} \label{est_linproblem}
 \begin{aligned}
 |\vecc{\Psi}|^2_{\mathcal{E}(t)}+|\vecc{\Psi}|^2_{\mathcal{D}(t)} \lesssim |\vecc{\Psi}_0|^2_{\mathcal{E}(0)}+\|f\|^2_{L^1(0,t; H^1(\R^n))}, \qquad t \in [0,T],
 \end{aligned}
 \end{equation}
with the energy semi-norms $|\cdot|_{\mathcal{E}(t)}$ and $|\cdot|_{\mathcal{D}(t)}$ defined in \eqref{EnergyNorm} and \eqref{DissipativeEnergyNorm}, respectively.
 \end{proposition}
\begin{proof}
The existence and regularity in the case $F = \boldsymbol{0}$ follow by~\cite[Corollary 2.6]{Bounadja_Said_2019}. The proof is based on the operator $\mathcal{A}$ being the infinitesimal generator of a $C_0$ semigroup of contraction on $\mathcal{H}^1$. The general case $F \neq \boldsymbol{0}$ follows by relying on standard semigroup results; see, for example,~\cite[Theorem 2.4.1 and Corollary 2.4.1]{zheng2004nonlinear}.\\
\indent We can derive the estimate by employing similar energy arguments to the ones of Section~\ref{Sec:EnergyEstimates}, where now the functionals $R^{(1)}$ and $R^{(2)}$ are given by
\begin{equation} 
\begin{aligned}
R^{(1)}(\varphi)= \tau (f, \varphi)_{L^2}, \quad R^{(2)}(\varphi)= \tau (\nabla f, \nabla \varphi)_{L^2}.
\end{aligned}
\end{equation}
This approach first leads to
	\begin{equation} \label{est_lin_f}
\begin{aligned}
|\vecc{\Psi}|^2_{\mathcal{E}(t)}+|\vecc{\Psi}|^2_{\mathcal{D}(t)} \lesssim |\vecc{\Psi}_0|^2_{\mathcal{E}(0)}+|\vecc{\Psi}|_{\mathcal{E}(t)}\|f\|_{L^1 H^1}, \ t \in [0,T].
\end{aligned}
\end{equation}
An application of Young's inequality then results in \eqref{est_linproblem}.
\end{proof}
 \subsection{Short-time existence for the nonlinear problem}
By relying on Proposition~\ref{Prop:LinExistence}, we can prove that a unique solution to our problem exists in a sufficiently brief interval of time. 
 \begin{theorem}\label{Local_Existence_1}
	Let $b> \tau c^2 > \tau c^2_g$ and $k \in \R$. Assume that $(\psi_0, \psi_1, \psi_2) \in H^2(\R^n) \times H^2(\R^n) \times H^1(\R^n)$, where $n \in \{2, 3\}$. 
	Then there exists a final time \[T=T( |\vecc{\Psi}_0|_{\mathcal{E}(0)}, \|\psi_0\|_{L^2})\]
such that the problem  given by \eqref{Main_System},  \eqref{Main_System_IC} has a unique solution \[\vecc{\Psi}=(\psi, v ,w, \eta)^T \in X=C^1([0,T]; \mathcal{H}^1) \cap C([0,T]; D(\mathcal{A})).\] Furthermore, the solution satisfies the following energy inequality:
	\begin{equation} \label{energy_est_nl}
	\begin{aligned}
	|\vecc{\Psi}|^2_{\mathcal{E}(t)}+|\vecc{\Psi}|^2_{\mathcal{D}(t)} 
\lesssim&\,  |\vecc{\Psi}|^2_{\mathcal{E}(0)}+ |\vecc{\Psi}|_{\mathcal{E}(t)}|\vecc{\Psi}|^2_{\mathcal{D}(t)}, \quad t \in [0,T].
	\end{aligned}
	\end{equation}
\end{theorem}
\begin{proof}
We intend to prove the statement by using the Banach fixed-point theorem, following standard techniques in nonlinear acoustics; see, e.g.,~\cite{Racke_Said_2019, KaltenbacherNikolic,lasiecka2017global}. We first need to introduce a suitable mapping. \\
\indent As already stated, for a given $\vecc{\Phi}=(\psi^\phi, v^\phi, w^\phi, \eta^\phi)^T$ in an appropriately chosen ball $\mathcal{B}_L$, we consider the mapping $\mathcal{T}: \vecc{\Phi} \mapsto \vecc{\Psi}$,
where $\vecc{\Psi}$ solves 
the following inhomogeneous linear problem:
\begin{equation} \label{IP_Psi0}
\begin{cases}
\partial_t 	\vecc{\Psi}-\mathcal{A}\vecc{\Psi} =\mathcal{F}(\vecc{\Phi}),\\
\vecc{\Psi} \vert_{t=0}=\vecc{\Psi}_0,
\end{cases}
\end{equation}
with the functional $\mathcal{F}$ defined in \eqref{def_F}. To choose a suitable space for $\vecc{\Phi}$,  we expect based on the linear existence theory and our previous energy arguments that it is a subspace of $C([0,T];  \mathcal{H}^2)$, where $\mathcal{H}^2$ is defined in \eqref{Hilbert_spaces}. Additionally, in order to use  Proposition~\ref{Prop:LinExistence}, we need to have $\mathcal{F}(\vecc{\Phi}) \in C^1([0,T]; \mathcal{H}^1) \cap C([0,T]; D(\mathcal{A}))$. This condition is equivalent to \[v^\phi w^\phi \in C^1([0,T]; L^2(\R^n)) \cap C([0,T]; H^1(\R^n)).\]  Motivated by this, we introduce the ball
	\begin{equation} \label{B_L}
\begin{aligned}
\mathcal{B}_{L}=\{\vecc{\Phi}=& \,(\psi^\phi, v^\phi, w^\phi, \eta^\phi)^T \in  C([0,T];  \mathcal{H}^2)\,:\\ \  &  \,  \|\vecc{\Phi}\|_{C(\mathcal{H}^2)}
+\sup_{t \in [0,T]}\|v^\phi_t(t)\|_{H^1}+\sup_{t \in [0,T]}\|w^\phi_t(t)\|_{L^2} \leq L,   \ \vecc{\Phi}(0)=\vecc{\Psi}_0 \, \}.
\end{aligned}
\end{equation}
The associated norm is given by
\begin{equation}
\begin{aligned}
\|\vecc{\Phi}\|_{\mathcal{B}_L}=&\,\|\vecc{\Phi}\|_{C(\mathcal{H}^2)}+\sup_{t \in [0,T]}\|v_t^\phi(t)\|_{H^1}+\sup_{t \in [0,T]}\|w_t^\phi(t)\|_{L^2}\\
=&\, |\vecc{\Phi}|_{\mathcal{E}(T)}+\sup_{t \in [0,T]}\|\psi^\phi(t)\|_{L^2}+\sup_{t \in [0,T]}\|v_t^\phi(t)\|_{H^1}+\sup_{t \in [0,T]}\|w_t^\phi(t)\|_{L^2}.
\end{aligned}
\end{equation} 
The radius $L \geq |\vecc{\Psi}_0|_{ \mathcal{E}_0}+\|\psi_0\|^2_{L^2}$ of the ball will be conveniently chosen as large enough below. The set $\mathcal{B}_L$ is a closed subset of a complete metric space $C([0,T]; \mathcal{H}^2)$ with the metric induced by the norm $\|\cdot\|_{\mathcal{B}_L}$. This set is non-empty for sufficiently large $L$ thanks to the linear existence result from Proposition~\ref{Prop:LinExistence}.\\
\indent We split the rest of the proof into two parts: proving that $\mathcal{T}$ is a self-mapping and proving its contractivity.  \vspace*{2mm}
\paragraph{\bf The self-mapping property} We focus first on proving that $\mathcal{T}(\mathcal{B}_L) \subset \mathcal{B}_L$. Take $\Phi \in \mathcal{B}_L$. We know that $\mathcal{F}(\Phi)=[0, 0, f, 0]^T $, where $f= \frac{2k}{\tau} v^\phi w^\phi$. We can directly check that
\begin{equation}
\begin{aligned}
& \sup_{t \in [0,T]}(\|f(t)\|_{L^2}+\|f_t(t)\|_{L^2}) \\
\lesssim&\, \sup_{t \in [0,T]} \left(\|v^\phi(t)\|_{L^\infty}\|w^\phi(t)\|_{L^2}+\|v_t^\phi(t)\|_{L^4}\|w^\phi(t)\|_{L^4}+\|v^\phi(t)\|_{L^\infty}\| w_{t}^\phi(t)\|_{L^2} \right).
\end{aligned}
\end{equation}
Therefore, we immediately have
\[
\|F\|_{C^1(\mathcal{H}^1)}+\|F\|_{C(D(\mathcal{A}))} =\|f\|_{C^1(L^2)}+\|f\|_{C(H^1)} 
\lesssim \,\|\Phi\|^2_{\mathcal{B}_L}< + \infty.\]
By taking into account also the regularity assumptions on the initial data, we conclude that problem  \eqref{IP_Psi0} has a unique solution $\vecc{\Psi} \in X= C^1([0,T]; \mathcal{H}) \cap C([0,T]; D(\mathcal{A}))$ on account of Proposition~\ref{Prop:LinExistence}. Thus our mapping is well-defined and it maps $\mathcal{B}_L$ into the space $X$. \\
\indent To show $\|\vecc{\Psi}\|_{\mathcal{B}_L} \leq L$, we rely on the energy estimate \eqref{energy_est_nl} from Proposition~\ref{Prop:LinExistence}.  We have
\begin{equation} 
\begin{aligned}
|\vecc{\Psi}|^2_{\mathcal{E}(t)}+|\vecc{\Psi}|^2_{\mathcal{D}(t)} \lesssim \, |\vecc{\Psi}_0|^2_{\mathcal{E}(0)}+\|f\|^2_{L^1(H^1)} \lesssim& \, |\vecc{\Psi}_0|^2_{\mathcal{E}(0)}+T^2\|\Phi\|^4_{\mathcal{B}_L} \\
\lesssim& \, |\vecc{\Psi}_0|^2_{\mathcal{E}(0)}+T^2L^4. 
\end{aligned}
\end{equation}
By observing that $\psi(t)=\int_0^t \psi_t(s) \, \textup{d}s + \psi_0$, we find that 
\begin{equation}
\begin{aligned}
\sup_{t \in [0,T]} \|\psi(t)\|^2_{L^2} \lesssim T^2 |\vecc{\Psi}|^2_{\mathcal{E}(t)}+\|\psi_0\|^2_{L^2} \lesssim T^2 (|\vecc{\Psi}_0|^2_{\mathcal{E}(0)}+T^2L^4) +\|\psi_0\|^2_{L^2}.
\end{aligned}
\end{equation}
Moreover, we have the identity
\begin{equation}
\begin{aligned}
 \sup_{t \in [0,T]} \|w_{t}(t)\|^2_{L^2}
= \sup_{t \in [0,T]} \left\|-\tfrac{1}{\tau} w(t)+ \tfrac{c^2_g}{\tau} \Delta \psi(t)+\tfrac{b}{\tau} \Delta v(t)+\tfrac{1}{\tau}\displaystyle \int_0^\infty g(s)\Delta \eta(s) \ds \right\|^2_{L^2},
\end{aligned}
\end{equation}
which implies
\begin{equation} \label{est_wt}
\begin{aligned}
\sup_{t \in [0,T]} \|w_{t}(t)\|^2_{L^2}\lesssim\, |\vecc{\Psi}|^2_{\mathcal{E}(t)}\lesssim \, |\vecc{\Psi}_0|^2_{\mathcal{E}(0)}+T^2L^4.
\end{aligned}
\end{equation}
We also know that 
\begin{equation} \label{est_vt}
\sup_{t \in [0,T]} \|v_{t}(t)\|^2_{H^1}=\sup_{t \in [0,T]} \|w(t)\|^2_{H^1} \lesssim\, |\vecc{\Psi}|^2_{\mathcal{E}(t)}.
\end{equation}
 Altogether, there exists a positive constant $C_\star$ such that
\begin{equation} \label{smalness_nl_eq_Rn}
\begin{aligned}
\|\vecc{\Psi}\|^2_{\mathcal{B}_L} \leq& \, C_\star (T^2+1)( |\vecc{\Psi}_0|^2_{\mathcal{E}(0)}+\|\psi_0\|^2_{L^2}+T^2L^4).
\end{aligned}
\end{equation}
We can then choose the final time $T$ small enough and the radius $L$ large enough so that $\vecc{\Psi} \in \mathcal{B}_L$. Indeed,  
for $L^2_0=|\vecc{\Psi}_0|^2_{\mathcal{E}(0)}+\|\psi_0\|^2_{L^2}$, we have
 \begin{equation}
\|\vecc{\Psi}\|^2_{\mathcal{B}_L}\leq C_\star L_0^2+C_\star T^2 (L_0^2+L^4+T^2L^4).
\end{equation}
We first fix $L$ large enough such that 
\begin{eqnarray*}
C_\star L_0^2\leq \frac{L^2}{2}. 
\end{eqnarray*}
Once $L$ is fixed, we can choose $T>0$ small enough such that 
\begin{equation} \label{smallness_T}
T^2\leq \min \left(1,\frac{L^2}{2C_\star (L^2_0+2L^4)}\right).
\end{equation}
By doing so, we obtain 
\[\|\vecc{\Psi}\|^2_{\mathcal{B}_L} \leq L^2.\]
Therefore, we conclude that $\mathcal{T}(\vecc{\Phi}) \in \mathcal{B}_L$ for this choice of the radius $L$ and the final time $T$.

 \vspace{2mm}
\paragraph{\bf Contractivity} To show contractivity, we take $\vecc{\Phi}, \vecc{\Phi}^\star \in \mathcal{B}_L$ and
\[ \mathcal{T}(\vecc{\Phi})=\vecc{\Psi} \quad \text{and} \quad \mathcal{T}(\vecc{\Phi}^\star)=\vecc{\Psi}^\star.
\] We have
\[
\mathcal{T}(\vecc{\Phi})-\mathcal{T}(\vecc{\Phi}^\star)=\vecc{\Psi}-\vecc{\Psi}^\star.
\] 
We can see the difference $\vecc{\mathcal{W}}=\vecc{\Psi}-\vecc{\Psi}^\star$ as a solution of the inhomogeneous problem with zero initial data:
\begin{equation}
\begin{aligned}
\begin{cases}
\partial_t \vecc{\mathcal{W}}-\mathcal{A}\vecc{\mathcal{W}}=\mathcal{F}(\vecc{\Phi})-\mathcal{F}(\vecc{\Phi}^\star), \\
\vecc{\mathcal{W}}|_{t=0}=0.
\end{cases}
\end{aligned}
\end{equation}
Let $\vecc{\Phi}^\star=(\psi^{\phi}_\star, v^{\phi}_\star, w^{\phi}_\star, \eta^{\phi}_\star)^T$. The right-hand side of the above problem is given by
\begin{equation}
\begin{aligned}
\mathcal{F}(\vecc{\Phi})-\mathcal{F}(\vecc{\Phi}^\star)= \dfrac{2k}{\tau}(0,\, 0,\,  v^{\phi} w^\phi-v^\phi_{\star} w_{\star}^\phi,\, 0)^T.
\end{aligned}
\end{equation}
Then by relying on the energy bound \eqref{energy_est_nl} from Proposition~\ref{Prop:LinExistence}, we directly obtain the estimate
\begin{equation} \label{est_contractivity}
\begin{aligned}
|\vecc{\mathcal{W}}|^2_{\mathcal{E}(t)}+|\vecc{\mathcal{W}}|^2_{\mathcal{D}(t)} \lesssim&\, \|v^{\phi} w^\phi-v^\phi_{\star} w_{\star}^\phi\|^2_{L^1H^1} \\
\lesssim&\, \|(v^{\phi}-v^\phi_{\star}) w^\phi+v^\phi_{\star}(w^\phi-w_{\star}^\phi)\|^2_{L^1H^1}.
\end{aligned}
\end{equation}
From here we have
\begin{equation}
\begin{aligned}
|\vecc{\mathcal{W}}|^2_{\mathcal{E}(t)}+|\vecc{\mathcal{W}}|^2_{\mathcal{D}(t)} \lesssim\, T^2 \left(|\vecc{\Phi}|^2_{\mathcal{E}(t)}+|\vecc{\Phi}^\star|^2_{\mathcal{E}(t)} \right) |\vecc{\Phi} -\vecc{\Phi}^\star|^2_{\mathcal{E}(t)} 
\lesssim\, T^2 L^2 |\vecc{\Phi} -\vecc{\Phi}^\star|^2_{\mathcal{E}(t)} .
\end{aligned}
\end{equation}
Denote $\vecc{\mathcal{W}}=(\psi-\psi_\star,v-v_\star,w-w_\star, \eta-\eta_\star)^T$. Similarly to before, we can derive the bound
\begin{equation}
\begin{aligned}
\sup_{t \in [0,T]} \|\psi(t)-\psi_{\star}(t)\|^2_{L^2} \lesssim T^2 |\vecc{\Psi}-\vecc{\Psi}^\star|^2_{\mathcal{E}(t)}=T^2|\vecc{\mathcal{W}}|^2_{\mathcal{E}(t)},
\end{aligned}
\end{equation}
as well as the estimate
\begin{equation}
\begin{aligned}
\sup_{t \in [0,T]} \|v_t(t)-v_{t, \star}(t)\|^2_{H^1}+ \sup_{t \in [0,T]} \|w_t(t)-w_{t, \star}(t)\|^2_{L^2} \lesssim |\vecc{\mathcal{W}}|^2_{\mathcal{E}(t)}.
\end{aligned}
\end{equation}
Altogether, we have
\begin{equation}
\begin{aligned}
\|\vecc{\mathcal{W}}\|^2_{\mathcal{B}_L} \lesssim &\,T^2(1+T^2) \left(|\vecc{\Phi}|^2_{\mathcal{E}(t)}+|\vecc{\Phi}^\star|^2_{\mathcal{E}(t)} \right) \|\vecc{\Phi} -\vecc{\Phi}^\star\|^2_{\mathcal{B}_L}\\
\lesssim&\, T^2(1+T^2) L^2 \|\vecc{\Phi} -\vecc{\Phi}^\star\|^2_{\mathcal{B}_L}.
\end{aligned}
\end{equation}
Therefore, we can guarantee that the mapping $\mathcal{T}$ is strictly contractive by reducing the final time $T$. An application of Banach's fixed-point theorem then yields a unique solution $\vecc{\Psi}=\vecc{\Phi} \in \mathcal{B}_L$.\\
\paragraph{\bf Unique solvability} Since $\mathcal{T}$ maps $\mathcal{B}_L$ into $X$, we conclude that, in fact, \[\vecc{\Psi} \in X=C^1([0,T]; \mathcal{H}^1) \cap C([0,T]; D(\mathcal{A})).\] It remains to prove uniqueness. For any two solutions $\vecc{\Psi}$ and $\vecc{\Psi}^\star$, we can prove similarly to \eqref{est_contractivity} that
\begin{equation}
\begin{aligned}
|\vecc{\Psi}(t) - \vecc{\Psi}^\star(t)|^2_{E}
\lesssim&\, \|(v-v_\star) w+v_\star (w-w_\star)\|^2_{L^1H^1}\\
\lesssim&\,  T\int_0^t \left(|\vecc{\Psi}(s)|^2_{E}+|\vecc{\Psi}^\star(s)|^2_{E} \right)|\vecc{\Psi}(s) -\vecc{\Psi}^\star(s)|^2_{E}\, \textup{d}s, \\
\end{aligned}
\end{equation}
for $t \in [0, T]$, where the semi-norm is given by 
\[|\vecc{\Psi}(t)|^2_{E}= \mathscr{E}_1[\vecc{\Psi}](t)+\mathscr{E}_2[\vecc{\Psi}](t)+\Vert w(t )\Vert _{L^{2}}^{2},\]
for energies $\mathscr{E}_1$ and $\mathscr{E}_2$ defined in \eqref{E_1_Eqv} and \eqref{E_2_Eqv}, respectively.  Then by Gronwall's inequality, we have $|\vecc{\Psi}(t) - \vecc{\Psi}^\star(t)|_{E}=0$. By combining this with the fact that $\psi(t)-\psi_\star(t)= \int_0^t (\psi_t(s)-\psi_{\star,t}(s))\, \textup{d}s$, we obtain $(\vecc{\Psi} - \vecc{\Psi}^\star)(t)=0$ at all times $t \in [0, T]$. This concludes the proof.
\end{proof}
\begin{remark}[On global solvability]
Due to the hard restriction \eqref{smallness_T} on final time, we cannot expect to obtain global solvability of the JMGT equation based on this result. The main issue is that we had to use the estimate
\[\|\psi(t)\|_{L^2} \lesssim \sqrt{T} \|\psi_t\|_{L^2L^2}  + \|\psi_0\|_{L^2}\]
to control $\|\psi(t)\|_{L^2}$ because we do not have Poincar\'e's inequality at our disposal. A way of resolving this problem is to consider acoustic velocity potentials in homogeneous spaces $\dot{H}^1(\R^n)$. However, this means that we have to restrict our setting to $n>2$ to work in Hilbert spaces.
\end{remark}
 \section{Global solvability in $\R^3$} \label{Sec:GlobalExistence}
To obtain global solvability for small data, we first have to modify the local existence result by working with acoustic potentials in $\dot{H}^1(\R^n)$.\\
\indent As already mentioned, the  homogeneous Sobolev space $\dot{H}^{1}(\R^n)$ is a Hilbert space if and only if $n>2$; see~\cite[Proposition 1.34]{bahouri2011fourier}. For this reason, we restrict ourselves to the physically most relevant setting $n=3$ to show global well-posedness and later suitable energy decay. \\
\indent We recall how the Hilbert space $\dot{\mathcal{H}}^1$ is defined in \eqref{dotHilbert_space_1} and also introduce the domain of the operator $\mathcal{A}$ as
	\begin{equation} \label{D(A)_R3}
	\begin{aligned}
	D(\mathcal{A})
	=\,  \left\{(\psi, v, w, \eta)^T \in \dot{\mathcal{H}}^{1} \left\vert\rule{0cm}{1cm}\right. \begin{matrix}
	w \in H^{1}(\R^n), \\[2mm]
	\dfrac{c^2_g}{\tau} \Delta \psi+\dfrac{b}{\tau} \Delta v+\dfrac{1}{\tau}\displaystyle \int_0^\infty g(s)\Delta \eta(s)  \in L^2(\R^n), \\[4mm]
	\eta \in D(\mathbb{T}) \end{matrix} \right\}.
	\end{aligned}
	\end{equation}
We first restate the linear existence result in $\R^3$ using the homogeneous Sobolev spaces.
 \begin{proposition} \label{Prop:LinExistence_R3}
	Let $b> \tau c^2> \tau c^2_g$ and let the final time $T>0$ be given. Assume that $(\psi_0, \psi_1, \psi_2) \in 
	\{\psi_0 \in \dot{H}^1(\R^3): \Delta \psi_0 \in L^2(\R^3)\} \times H^2(\R^3) \times H^1(\R^3)$ and that the source term is given by \[F=[0, 0, f, 0]^T \in C^1([0,T]; \mathcal{H}^1) \cap C([0,T]; \mathcal{D}(\mathcal{A})).\] Then the linear initial-value problem 
	\begin{equation} \label{IP_F}
	\begin{cases}
	\partial_t \vecc{\Psi}-\mathcal{A} \vecc{\Psi} = F, \\
	\vecc{\Psi} \vert_{t=0}=\vecc{\Psi}_0
	\end{cases}
	\end{equation}
 has a unique solution 
$\vecc{\Psi} \in C^1([0,T]; \dot{\mathcal{H}}^1) \cap C([0,T]; D(\mathcal{A}))$.
	Furthermore, the following estimate holds:
	\begin{equation} \label{bound_lin}
	\begin{aligned}
 \begin{aligned}
\|\vecc{\Psi}\|^2_{\mathcal{E}(t)}+|\vecc{\Psi}|^2_{\mathcal{D}(t)} \lesssim \|\vecc{\Psi}_0\|^2_{\mathcal{E}(0)}+\|f\|^2_{L^1 H^1}, \ t \in [0,T],
\end{aligned}
	\end{aligned}
	\end{equation}
where $\|\cdot\|_{\mathcal{E}(t)}=|\cdot|_{\mathcal{E}(t)}$ and $|\cdot|_{\mathcal{D}(t)}$ are defined in \eqref{EnergyNorm} and \eqref{DissipativeEnergyNorm}, respectively.
\end{proposition}
\begin{proof}
The proof follows the same steps of the proof of Proposition~\ref{Prop:LinExistence}, based on the operator $\mathcal{A}$, with $D(\mathcal{A})$ defined in \eqref{D(A)_R3}, being the infinitesimal generator of a $C_0$ semigroup of contraction on $\dot{\mathcal{H}}^1$. 
\end{proof}
\noindent Next we can re-state the nonlinear local existence result in $\R^3$. 
\begin{theorem} \label{Thm:LocalExistence_3D}
Let $b> \tau c^2> \tau c^2_g$ and $k \in \R$. Assume that
\[(\psi_0, \psi_1, \psi_2) \in \{\psi_0: \psi_0 \in \dot{H}^1(\R^3), \ \Delta \psi_0 \in L^2(\R^3) \} \times H^2(\R^3) \times H^1(\R^3).\]
Then there exists a final time \[T=T( |\vecc{\Psi}_0|^2_{\mathcal{E}(0)})\]
such that problem   \eqref{Main_System},  \eqref{Main_System_IC} has a unique solution 
\[\vecc{\Psi}=(\psi, v ,w, \eta)^T \in C^1([0,T]; \dot{\mathcal{H}}^1)\cap C([0,T]; D(\mathcal{A})).\] The solution of the problem satisfies the energy estimate
\begin{equation} \label{energy_bound_Thm2}
\begin{aligned}
	\|\vecc{\Psi}\|^2_{\mathcal{E}(t)}+|\vecc{\Psi}|^2_{\mathcal{D}(t)} 
\lesssim&\,  \|\vecc{\Psi}\|^2_{\mathcal{E}(0)}+ \|\vecc{\Psi}\|_{\mathcal{E}(t)}|\vecc{\Psi}|^2_{\mathcal{D}(t)}, \quad t \in [0,T].
\end{aligned}
\end{equation}
\end{theorem}
\begin{proof}
 The proof follows along the same lines as before, with the difference that now $|\cdot|_{\mathcal{E}(T)}$ defines a norm in $C([0,T]; \dot{\mathcal{H}}^2)$, where the Hilbert space $\dot{\mathcal{H}}^2$ is defined in \eqref{dotHilbert_space_2}.  We can, therefore, define the ball in $C([0,T]; \dot{\mathcal{H}}^2)$ as
\begin{equation} \label{B_L_R3}
\begin{aligned}
\mathcal{B}_{L}=\{\vecc{\Phi}=& \,(\psi^\phi, v^\phi, w^\phi, \eta^\phi)^T \in  C([0,T];  \dot{\mathcal{H}}^2):   \,  \|\vecc{\Phi}\| _{\mathcal{B}_L} \leq L,  \ \vecc{\Phi}(0)=\vecc{\Psi}_0 \, \}.
\end{aligned}
\end{equation}
but this time supplemented with the norm 
\begin{equation} \label{norm_BL_R3}
\|\vecc{\Phi}\| _{\mathcal{B}_L}= \|\vecc{\Phi}\|_{\mathcal{E}(T)}+|\vecc{\Psi}|_{\mathcal{D}(T)}+\sup_{t \in [0,T]}\|v^\phi_{t}(t)\|_{H^1}+\sup_{t \in [0,T]}\|w^\phi_{t}(t)\|_{L^2}.
\end{equation}
When proving that $\mathcal{T}(\mathcal{B}_L) \subset \mathcal{B}_L$, the bound \eqref{smalness_nl_eq_Rn} changes to
\begin{equation} 
\begin{aligned}
\|\vecc{\Psi}\|^2_{\mathcal{B}_L} \leq& \, C_\star ( \|\vecc{\Psi}_0\|^2_{\mathcal{E}(0)}+T^2L^4).
\end{aligned}
\end{equation}
We can thus guarantee that $\|\vecc{\Psi}\|_{\mathcal{B}_L} \leq L$ by choosing the radius large enough and then the final time small enough so that
\begin{equation} \label{smallness_T_new}
 C_\star \|\vecc{\Psi}_0\|^2_{\mathcal{E}(0)} \leq \frac12 L^2, \qquad T^2 \leq \frac{1}{2 C_\star L^2}.
\end{equation}
The rest of the proof follows along the same lines as before. We omit the details here.
\end{proof}  
\noindent From \eqref{smallness_T_new}, it is clear that we can increase $T$ by taking smaller data. We prove this claim next.
\begin{theorem} \label{Thm:GlobalExistence}
Let $b> \tau c^2> \tau c^2_g$ and $k \in \R$. Assume that
\begin{align} \label{initial_data_global}
	(\psi_0, \psi_1, \psi_2) \in \{\psi_0: \psi_0 \in \dot{H}^1(\R^3), \ \Delta \psi_0 \in L^2(\R^3) \} \times H^2(\R^3) \times H^1(\R^3).
\end{align}	
 Then there exists small $\delta>0$  such that if 
\begin{equation} \label{data_smallness3D}
\|\vecc{\Psi}_0\|^2_{\mathcal{E}(0)} \leq \delta,
\end{equation}
then problem \eqref{Main_System},  \eqref{Main_System_IC} has a global solution
\[\vecc{\Psi} \in \{\vecc{\Psi}=(\psi, v ,w, \eta)^T:\,  \vecc{\Psi} \in C([0, \infty); \dot{\mathcal{H}}^2),  \, (v, w) \in C^1((0, +\infty); H^1(\R^3) \times L^2(\R^3)) \}.\]
\end{theorem}
\begin{proof}
 Because of the term $-b\Delta_t u$ in equation \eqref{Main_Equation} and the type of nonlinearity in the model, we can prove the global existence without appealing to the decay of the linearized problem. Let $T>0$ be the maximal time of local existence given by Theorem \ref{Thm:LocalExistence_3D}. Our goal is to prove by a continuity argument that the norm 
\begin{eqnarray*}
|||\vecc{\Psi} |||_{(0,t)}=\| \vecc{\Psi} \|_{\mathcal{E}(t)}+|\vecc{\Psi} |_{\mathcal{D}(t)}
\end{eqnarray*}
is  uniformly bounded for all time if the initial energy $|\vecc{\Psi}_0|^2_{\mathcal{E}(0)} $ is sufficiently small. Note that thanks to estimates \eqref{est_wt} and \eqref{est_vt}, we know that 
\[ \|\vecc{\Psi}\|_{\mathcal{B}_L(0,t)} \lesssim  \|\vecc{\Psi} \|_{\mathcal{E}(t)}+|\vecc{\Psi} |_{\mathcal{D}(t)}=|||\vecc{\Psi} |||_{(0,t)},\] where the norm $\|\cdot\|_{\mathcal{B}_L(0,t)}$ is defined as in \eqref{norm_BL_R3}, only with the time interval $[0,T]$ replaced by $[0,t]$. Theorem~\ref{Thm:LocalExistence_3D} provides us with the energy bound
\begin{equation}
\begin{aligned}
\|\vecc{\Psi}\|^2_{\mathcal{E}(t)}+|\vecc{\Psi}|^2_{\mathcal{D}(t)} 
\lesssim&\,  \|\vecc{\Psi}\|^2_{\mathcal{E}(0)}+ \|\vecc{\Psi}\|_{\mathcal{E}(t)}|\vecc{\Psi}|^2_{\mathcal{D}(t)}, \quad t \in [0,T].
\end{aligned}
\end{equation}
This implies that for all $t\in[0,T]$, 
\begin{eqnarray}\label{Main_Y_Estimate}
|||\vecc{\Psi}|||_{(0,t)}\leq \|\vecc{\Psi}_0\|_{\mathcal{E}(0)}+C |||\vecc{\Psi}|||_{(0,t)}^{3/2},
\end{eqnarray}
On account of Lemma \ref{Lemma_Stauss}, the above inequality implies that there exists a positive constant $C$, independent of $t$, such that
\begin{eqnarray*}
|||\vecc{\Psi}|||_{(0,t)} \leq C. 
\end{eqnarray*}
This uniform bound guarantees that our local solution can be continued to $T=\infty$. 
\end{proof}
\noindent Accordingly, the JMGT equation in hereditary media with initial data \eqref{initial_data_global} admits a unique solution $\psi$ such that
\begin{equation}
\begin{aligned}
& \psi \in C([0, +\infty); \{\phi \in \dot{H}^1(\R^3):\, \nabla^2 \phi \in L^2(\R^3)\})\, \cap \, C^1([0, +\infty); \dot{H}^1(\R^3)),\\
& \psi_t \in  C([0, +\infty); H^2(\R^3)) \, \cap \, C^1([0, +\infty); H^1(\R^3)), \\
& \psi_{tt} \in C([0, +\infty); H^1(\R^3))\, \cap  \, C^1([0, +\infty); L^2(\R^3)).
\end{aligned}
\end{equation}
\section{Decay rates for the JMGT equation in $\R^3$} \label{Sec:DecayRates}
We next wish to see if and how the solution of \eqref{Main_Equation} decays with time. To answer these questions, we first need to derive new decay estimates for $v=\psi_t$ in the linearized model. 
\subsection{Decay estimates for the linearized system}
The corresponding linear problem is given by the system
 \begin{equation}  \label{Main_System_linear}
 \begin{cases}
 \psi_{t}=v, \\ 
 v_{t}=w, \\ 
 \tau w_{t}=- w+c^2_g\Delta \psi+b\Delta v + \displaystyle%
 \int_{0}^\infty g(s)\Delta\eta(s)\ds,
 \\ 
 \eta_{t}=v-\eta_{s},%
 \end{cases}%
 \end{equation}
 supplemented with the same initial data \eqref{Main_System_IC}.  To formulate the result, we introduce the vector
 \[\vecc{U}=(v+\tau w,\nabla( \psi+\tau v),\nabla v),\]
and the corresponding initial vector $\vecc{U}_0=(\psi_1+\tau \psi_2, \nabla(\psi_0+ \tau \psi_1), \nabla \psi_1)$. The decay rates for $\vecc{U}$ are given by the following three results.
 \begin{lemma}[see Theorem 3.1 in~\cite{Bounadja_Said_2019}]\label{Lemma_Decay_Linear}
 Let $s\geq 0$ be an integer. Assume that $\vecc{U}_0\in L^1(\R^n) \cap H^s(\R^n)$, where $n \in \N$, and that $b>\tau c^2$. Then, for any $0 \leq j\leq s$, it holds that 
 \begin{align}  \label{Decay_Linearized}
 		\Vert\nabla^{j}\vecc{U}(t)\Vert _{L^{2}}\,\lesssim \Vert \nabla^j\vecc{\Psi}(t)\Vert_{\mathscr{E}_1}\lesssim (1+t)^{-{n}/{4-{j}/{2}}}\Vert \vecc{U}_{0}\Vert _{L^{1}} +e^{-\frac{\lambda}{2} t} \Vert\nabla^{j}\vecc{U}_{0}\Vert _{L^{2}},
 		\end{align}
 	where $\lambda$ is a positive constant independent of $t$, and
 	\begin{equation} \label{new_norm}
 	\begin{aligned}
 	\|\vecc{\Psi}(t)\|_{\mathscr{E}_1} =&\, \begin{multlined}[t]\Vert\nabla(\psi +\tau \psi_t)(t)\Vert^{2}_{L^2} + \Vert(\psi_t+\tau \psi_{tt})(t)\Vert^{2}_{L^2}+\Vert \nabla \psi_t(t)\Vert^{2}_{L^2}
\\ 	+\Vert \nabla\eta\Vert^{2}_{L^2, -g'}.\end{multlined}
 	\end{aligned}
 	\end{equation} 
 \end{lemma} 
 Estimate \eqref{Decay_Linearized} does not directly yield a decay rate for $\Vert \nabla^j \psi_t\Vert_{L^2}=\Vert \nabla^j v\Vert_{L^2}$. However, we can obtain it through the bound
   \begin{equation}\label{Ineq_L_2}
\Vert \nabla ^j v\Vert_{L^2}\lesssim \Vert \nabla^j(v+\tau w)\Vert_{L^2}+\Vert \nabla^j w\Vert_{L^2}
\end{equation} 
and \eqref{Decay_Linearized} if we have a decay rate for  $\Vert w\Vert_{L^2}$. This rate is the result of the next proposition. 
\begin{proposition}\label{Decay_w_New}
Let the assumptions of Lemma~\ref{Lemma_Decay_Linear} hold with $s \geq 1$ and let $w_0 \in H^s(\R^n)$. Then, for any $n\in \N$ and any $0 \leq j\leq s-1$, we have
\begin{eqnarray}\label{Decay_estimate_W} 
\Vert\nabla^{j}w(t)\Vert_{L^{2}}\lesssim (\Vert \nabla^j w_0\Vert_{L^2}+\Vert \vecc{U}_{0}\Vert_{L^{1}}+\Vert\nabla^{j+1}\vecc{U}_{0}\Vert_{L^{2}})(1+t)^{-\frac{n}{4}-\frac{j}{2}-\frac{1}{2}},
\end{eqnarray}
provided that the thermal relaxation $\tau>0$ is sufficiently small.
\end{proposition}   
 \begin{proof}
For proving the above estimate, we need to employ the decay rates of the Fourier transform of the solution; cf.~\cite{Bounadja_Said_2019}. Recall how the low-order energy $E_1$ is defined in \eqref{energy}. We then define
     \[\hat{E}_1(\xi,t)=\mathscr{F}(E_1(x,t)),\]  
 where $``\mathscr{F}"$ stands for the Fourier transform and the variable dual to $x$ is denoted by $\xi$. Then the following estimate holds:
 \begin{equation}\label{Eexp}
     \hat{E}_1(\xi,t)\lesssim \hat{E}_1(\xi,0)\exp{(-\lambda \tfrac{|\xi|^2}{1+|\xi|^2}t)}
     \end{equation} 
for all $t \geq 0$; cf.~\cite[Proposition 4.1]{Bounadja_Said_2019}. The constant $\lambda$ is positive and independent of $t$ and $\xi$. For the linearized problem, it is clear that estimate \eqref{E_0_Energy} holds with $R^{(1)}$ set to zero. In other words, we have
	\begin{equation} 
\begin{aligned}
\frac{1}{2}\frac{\textup{d}}{\dt}\int_{\mathbb{R}^{n}}\tau \left\vert  
w\right\vert ^{2}\dx+\frac{1}{2}\int_{\mathbb{R}^{n}} |w|^2 \dx
\lesssim \,
\Vert
\Delta (\psi+\tau v)\Vert_{L^2}^2+\Vert \Delta v\Vert _{L^{2}}^2+\Vert\Delta\eta\Vert^{2}_{L^2, g}.
\end{aligned}
\end{equation}
 Thus we know that
  \begin{equation}\label{w_Energy_Fourier}
	\begin{aligned}
	\frac{1}{2}\frac{\textup{d}}{\dt}\tau \left\vert \hat{w}\right\vert
	^{2}+ \frac{1}{2}|\hat{w}|^2 \lesssim \, |\xi|^2 \hat{E}_1(\xi,t).
	\end{aligned}
	\end{equation}
By plugging in estimate \eqref{Eexp} for $\hat{E}_1(\xi,t)$ in the above inequality, we obtain
\begin{equation}
\frac{\textup{d}}{\dt} \left\vert \hat{w}\right\vert
	^{2}\leq -\frac{1}{\tau} \left\vert \hat{w}\right\vert
	^{2}+C|\xi|^2 \hat{E}_1(\xi,0)\exp{\left(-\lambda \tfrac{|\xi|^2}{1+|\xi|^2}t\right)}.  
\end{equation} 
We can then apply the differential version of Gronwall's inequality to arrive at 
 \begin{equation}\label{w_Estimate_main}
 \begin{aligned}
\left\vert \hat{w}\right\vert
	^{2} \leq&\, |\hat{w}_0|^2 \exp{\left(-\tfrac{1}{\tau} t\right)} +C|\xi|^2\hat{E}_1(\xi,0)\int_0^t \exp{\left(-\lambda \tfrac{|\xi|^2}{1+|\xi|^2} s\right)}\, \exp{\left(-\tfrac{1}{\tau}(t-s)\right)}\ds,
\end{aligned}	
\end{equation}
which directly leads to	
 \begin{equation}\label{w_Estimate_main}
\begin{aligned}
\left\vert \hat{w}\right\vert
^{2} \leq& \begin{multlined}[t]
	 \, |\hat{w}_0|^2 \exp{\left(-\tfrac{1}{\tau} t\right)}\\ +C|\xi|^2\hat{E}_1(\xi,0)\exp{\left(-\tfrac{1}{\tau} t\right)}\left(\tfrac{1}{\tau}-\lambda \tfrac{|\xi^2|}{1+|\xi|^2}\right)^{-1} \left[\exp{\left(-(\lambda \tfrac{|\xi|^2}{1+|\xi|^2}-\tfrac{1}{\tau})t\right)} -1\right]. \end{multlined}
\end{aligned}	
\end{equation}
To further bound the right-hand side side, we can use the identity 
\[\frac{1}{\frac{1}{\tau}-\lambda \frac{|\xi^2|}{1+|\xi|^2}}=
\frac{\tau  \left(|\xi|^2+1\right)}{|\xi|^2 (1-\lambda  \tau )+1}.
\] 
Assuming that the thermal relaxation is small enough so that $\tau< \frac{1}{\lambda}$, it holds
\begin{equation}
\frac{1}{\frac{1}{\tau}-\lambda \frac{|\xi^2|}{1+|\xi|^2}}\leq \frac{\tau}{1-\lambda  \tau}.
\end{equation}
Altogether, for small $\tau>0$, we obtain 
\begin{equation}\label{w_Fourier_Estimate}
\left\vert \hat{w}\right\vert
	^{2}\leq |\hat{w}_0|^2 \exp{\left(-\tfrac{1}{\tau} t\right)}+C|\xi|^2 \hat{E}_1(\xi,0)\exp{\left(-\lambda \tfrac{|\xi|^2}{1+|\xi|^2}t\right)}.  
\end{equation}
We can use the estimate
\begin{equation} \label{ineq_U_E_1}
\hat{E}_1(\xi,0) \lesssim |\hat{\vecc{U}}(\xi,0)|^2,  
\end{equation}
where $\hat{\vecc{U}}(\xi,t)=\mathcal{F}(\vecc{U}(x,t))$; see~\cite[Lemma 4.3]{Bounadja_Said_2019}. By applying Plancherel's theorem and \eqref{ineq_U_E_1}, we find 
\begin{equation}\label{Plancherel}
\begin{aligned}
\Vert\nabla^{j}w(t)\Vert_{L^{2}}^{2}=& \, \int_{\R^{n}}|\xi|^{2j}|\hat{w}(\xi,t)|^{2}\, \textup{d}\xi\\
	\lesssim& \,\Vert \nabla^j w_0\Vert_{L^2}^2\exp{\left(-\tfrac{1}{\tau} t\right)}+ \int_{\R^{n}}|\xi|^{2(j+1)} \exp{\left(-\lambda \tfrac{|\xi|^2}{1+|\xi|^2}t\right)}|\hat{\vecc{U}}(\xi,0)|^{2}\, \textup{d}\xi 
\end{aligned}	
\end{equation}
for any $j\geq 0$.  The second term on the right-hand side of estimate \eqref{Plancherel} can be split into two terms as follows:
\begin{equation} \label{split Intg} 
\begin{aligned}
     &\int_{\R^{n}}|\xi|^{2(j+1)} \exp{\left(-\lambda \tfrac{|\xi|^2}{1+|\xi|^2}t\right)}|\hat{\vecc{U}}(\xi,0)|^{2}\textup{d} \xi\\
     =& \, \begin{multlined}[t]\int_{|\xi|\leq 1}|\xi|^{2(j+1)} \exp{\left(-\lambda \tfrac{|\xi|^2}{1+|\xi|^2}t\right)}|\hat{\vecc{U}}(\xi,0)|^{2}\textup{d} \xi \\
     + \int_{|\xi|\geq 1}|\xi|^{2(j+1)} \exp{\left(-\lambda \tfrac{|\xi|^2}{1+|\xi|^2}t\right)}|\hat{\vecc{U}}(\xi,0)|^{2}\textup{d}\xi.\end{multlined}
     	\end{aligned}
\end{equation}   
We can then use the bound
\begin{equation} \label{rho*}
\frac{|\xi|^2}{1+|\xi|^2}\geq 
\left \{ \begin{aligned}
&\dfrac{1}{2}|\xi|^{2}  \quad &&\text{if }\ |\xi|\leq 1, \\
&\dfrac{1}{2} \quad &&\text{if } \ |\xi|\geq 1. \end{aligned}
\right.
\end{equation}  	
Concerning the first integral on the right in \eqref{split Intg}, by exploiting the inequality \[\int_{0}^{1}r^{n-1}e^{-r^{2}t}\textup{d}r \leq C(n)(1+t)^{-{n}/{2}},\]
given in Lemma~\ref{Lemma:Ineq} together with \eqref{rho*}, we find that
\begin{equation}\label{I1}
\begin{aligned}
     	\int_{|\xi|\leq 1}|\xi|^{2j} \exp{\left(-\lambda \tfrac{|\xi|^2}{1+|\xi|^2}t\right)}|\hat{\vecc{U}}(\xi,0)|^{2}\textup{d} \xi\leq& \, \Vert\hat{\vecc{U}}_{0}\Vert_{L^{\infty}}^{2}\int_{|\xi|\leq 1}|\xi|^{2(j+1)} \exp{\left(-\tfrac{\lambda}{2} \tfrac{|\xi|^2}{1+|\xi|^2}t\right)}\, \textup{d}\xi\\
     	\lesssim& \, (1+t)^{-\frac{n}{2}-1-j}\Vert \vecc{U}_{0}\Vert_{L^{1}}^{2}.
\end{aligned}
\end{equation}
On the other hand, in the high-frequency region where $|\xi|\geq 1$, we have 
\begin{equation}\label{I2}
\begin{aligned}
\int_{|\xi|\geq 1}|\xi|^{2j} \exp{\left(-\lambda \tfrac{|\xi|^2}{1+|\xi|^2}t\right)}|\hat{\vecc{U}}(\xi,0)|^{2}\textup{d}\xi
\leq&\, e^{-\frac{\lambda}{2}t}\int_{|\xi|\geq 1}|\xi|^{2(j+1)} |\hat{\vecc{U}}(\xi,0)|^{2}d\xi \\
\leq& \, e^{-\frac{\lambda}{2}t}\Vert\nabla^{j+1}\vecc{U}_{0}\Vert_{L^{2}}^{2}.
\end{aligned}	
\end{equation} 
By plugging the above two estimates into inequality \eqref{Plancherel}, we finally obtain 
\begin{equation}
\Vert\nabla^{j}w(t)\Vert_{L^{2}}\lesssim e^{-\frac{1}{2\tau} t}\Vert \nabla^j w_0\Vert_{L^2}+C(1+t)^{-\frac{n}{4}-\frac{j}{2}-\frac{1}{2}}\Vert \vecc{U}_{0}\Vert_{L^{1}(\R^{n})}+ e^{-\frac{\lambda}{4}t}\Vert\nabla^{j+1}\vecc{U}_{0}\Vert_{L^{2}}. 
\end{equation}
This implies that estimate \eqref{Decay_estimate_W} holds for large $t$, which completes the proof.  
\end{proof} 
\noindent We are ready to prove the decay rate for $v=\psi_t$.	
	\begin{lemma}
Let the assumptions of Proposition~\ref{Decay_w_New} hold. Then, for any $n\in \N$ and any $0 \leq j\leq s-1$, we have
	\begin{equation}\label{v_L_2_Estimate}
\Vert\nabla^{j}v(t)\Vert_{L^{2}}\lesssim (\Vert \nabla^j w_0\Vert_{L^2}+\Vert \vecc{U}_{0}\Vert_{L^{1}}+\Vert\nabla^{j}\vecc{U}_{0}\Vert_{H^{1}})(1+t)^{-\frac{n}{4}-\frac{j}{2}}.
\end{equation}
Furthermore, assuming $\vecc{U}_0 \in L^1(\R^3) \cap H^3(\R^3)$ and $w_0 \in H^2(\R^3)$, it holds
\begin{equation}\label{v_L_infty_Estimate}
\Vert v(t) \Vert_{L^\infty}\lesssim (\Vert w_0\Vert_{H^2}+\Vert \vecc{U}_0\Vert_{L^1}+\Vert \vecc{U}_0\Vert_{H^3})(1+t)^{-\frac{n}{2}}. 
\end{equation}    
\end{lemma}
\begin{proof}
The estimate \eqref{v_L_2_Estimate} is a result of combining the bounds \eqref{Decay_Linearized}, \eqref{Decay_estimate_W}, and estimate \eqref{Ineq_L_2}. To prove the second estimate, we use the Gagliardo--Nirenberg interpolation inequality in the form of
	\begin{equation}\label{L_infty_Interp}
\left\Vert v\right\Vert _{L^{\infty }}\leq C\left\Vert \nabla ^{2}%
v\right\Vert _{L^{2}}^{\frac{n}{4}}\left\Vert v%
\right\Vert _{L^{2}}^{1-\frac{n}{4}}. 
\end{equation}
Taking into account estimate \eqref{v_L_2_Estimate} immediately yields \eqref{v_L_infty_Estimate}. \end{proof}
\subsection{Decay estimates for the nonlinear problem}   
We are now ready to prove decay estimates for the solution to the nonlinear problem. Similarly to before, we introduce the vector 
\[\vecc{U}=(v+\tau w,\nabla( \psi+\tau v),\nabla v),\]
where now $(\psi, v, w, \eta)$ solves the nonlinear problem.
\begin{theorem}
\label{Theorem_Decay} Let $b>\tau c^2> \tau c^2_g$ and $n=3$. Assume that the initial data $(\psi_0, \psi_1, \psi_2)$ satisfy the regularity and smallness assumptions \eqref{initial_data_global} and \eqref{data_smallness3D}. Furthermore, suppose that $\vecc{U}_0 = \vecc{U}(t=0)\in L^1(\R^3) \cap H^1(\R^3)$ and $w_0\in L^2(\R^3)$, and that  
\begin{equation} \label{Lambda_0}
\Lambda_0=\Vert w_0\Vert_{L^2}+ \Vert \vecc{U}_{0}\Vert _{L^{1}}+\Vert  \vecc{U}_{0}\Vert _{H^1}
\end{equation}
 is small enough.   Then, the global solution of \eqref{eta syst} satisfies the following decay rates:
\begin{equation} \label{decay_rates}
\begin{aligned}
&\, \Vert \nabla^j\vecc{U}(t)\Vert_{L^2}\lesssim  \,\Lambda_0(1+t)^{-\frac{n}{4}-\frac{j}{2}} \quad \text{for} \ j=0, 1,\\[1mm]
& \, \Vert v(t)\Vert_{L^2}\lesssim \, \Lambda_0(1+t)^{-\frac{n}{4}}, \\[1mm]
&\Vert w(t)\Vert_{L^2}\lesssim \, \Lambda_0(1+t)^{-\frac{n}{4}-\frac{1}{2}}.
\end{aligned}
\end{equation}
\end{theorem}
\begin{proof}
Let $\vecc{\Psi}=(\psi, v, w, \eta)^T $ be the global solution of our system according to Theorem~\ref{Thm:GlobalExistence}. 
We have
\[\Vert \vecc{U}(t)\Vert_{L^2} \leq \Vert \vecc{\Psi}(t)\Vert_{\mathscr{E}_1} \quad \text{and} \quad \Vert \nabla \vecc{U}(t)\Vert_{L^2} \leq \Vert \nabla \vecc{\Psi}(t)\Vert_{\mathscr{E}_1},\]
with the norm $\|\cdot\|_{\mathscr{E}_1}$ as in \eqref{new_norm}. Motivated by the decay estimates for the linearized problem obtained in Lemma~\ref{Lemma_Decay_Linear} and Proposition~\ref{Decay_w_New}, which we expect to hold for the nonlinear problem as well for small data, we define 
\begin{equation}\label{M_t_Decay}
\begin{aligned}
\mathcal{M}(t)=\sup_{0\leq \sigma\leq t}\Big[(1+\sigma)^{n/4}
\Vert  \vecc{U}(\sigma)\Vert_{L^2}+(1+\sigma)^{n/4+1/2}
\Vert \nabla \vecc{U}(\sigma)\Vert_{L^2}\Big.\\
\Big.+(1+\sigma)^{n/4}
\Vert  v(\sigma)\Vert_{L^2}+(1+\sigma)^{\frac{n}{4}+\frac{1}{2}}\Vert w(\sigma)\Vert_{L^2}\Big].
\end{aligned} 
\end{equation}
\noindent Keeping in mind the $L^\infty$ bound \eqref{v_L_infty_Estimate} for $v$,  we also introduce the quantity 
\begin{eqnarray}\label{M_0_t}
M_{0}(t) &=&\sup_{0\leq \sigma \leq t}(1+\sigma )^{3n/8}\left\Vert
v\left( \sigma \right) \right\Vert _{L^{\infty }}.  
\end{eqnarray}
By using the Gagliardo--Nirenberg interpolation inequality \eqref{L_infty_Interp}, we deduce that 
\begin{eqnarray}\label{M_0_M_Ineq}
M_0(t)\lesssim \mathcal{M}(t). 
\end{eqnarray}
The reason for taking the exponent $3n/8$ in \eqref{M_0_t} instead of $n/2$ is to make sure that the inequality above holds. The resulting slow decay of  $\Vert v\Vert_{L^\infty}$ is a consequence of the slow decay of $\Vert \nabla^2 v\Vert_{L^2}$ given by $(1+t)^{-n/4-1/2}$.  Despite this, we can still prove that the vector $\vecc{U}$ decays as fast as in the linear equation thanks to the fast decay of $\Vert w\Vert_{L^2}$.   \\
\indent Our next aim is to show that $\mathcal{M}(t)$ is bounded uniformly in $t$ if $\Lambda_0$, defined in \eqref{Lambda_0}, is small enough. We begin by writing the solution to our problem as
\begin{eqnarray*}
\vecc{\Psi}(t)=e^{t \mathcal{A}}\vecc{\Psi}_0+\displaystyle \int_0^t e^{(t-r)\mathcal{A}}\mathcal{F}(\vecc{\Psi})(r)\, \textup{d}r. 
\end{eqnarray*}
From here we directly estimate
\begin{equation} \label{decay_nl_1}
\begin{aligned}
\Vert \nabla^j\vecc{U}(t)\Vert_{L^2}\, \textcolor{mygreen}{\leq}\, \Vert \nabla^j \vecc{\Psi}(t)\Vert_{\mathscr{E}_1}\leq&\, \Vert \nabla^je^{t\mathcal{A}}\vecc{\Psi}_{0}\Vert_{\mathscr{E}_1}+\int_0^t \left\Vert \nabla^je^{(t-r
)\mathcal{A}}\mathcal{F}(\vecc{\Psi})(r)\right\Vert_{\mathscr{E}_1} \textup{d}r\\
=&\, \begin{multlined}[t]\Vert \nabla^j e^{t\mathcal{A}}\vecc{\Psi}_{0}\Vert_{\mathscr{E}_1}+\int_0^{t/2} \left\Vert \nabla^je^{(t-r
)\mathcal{A}}\mathcal{F}(\vecc{\Psi})(r)\right\Vert_{\mathscr{E}_1} \textup{d}r\\
+\int_{t/2}^{t} \left\Vert \nabla^j e^{(t-r
)\mathcal{A}}\mathcal{F}(\vecc{\Psi})(r)\right\Vert_{\mathscr{E}_1} \textup{d}r \end{multlined}
\end{aligned}
\end{equation} 
for $j \in \{0,1\}$. By applying the linear decay rate \eqref{Decay_Linearized} from Lemma~\ref{Lemma_Decay_Linear}, we have  
\begin{equation}\label{Decay_Liearized_Prob}
\begin{aligned}
\Vert \nabla^je^{t\mathcal{A}}\vecc{\Psi}_{0}\Vert_{\mathscr{E}_1}\lesssim&\, (1+t)^{-n/4-j/2}\left(\Vert \vecc{U}_0\Vert_{L^1}+\Vert \nabla^j   \vecc{U}_0\Vert_{L^2}\right),
\end{aligned}
\end{equation}
with $j \in \{0,1\}$. We need to estimate the remaining two integrals on the right-hand side of \eqref{decay_nl_1}.
For the first one, we have by  using the linear estimate \eqref{Decay_Linearized} and Duhamel's principle, 
\begin{equation}\label{Estimate_Decay_1}
\begin{aligned}
\int_0^{t/2} \left\Vert \nabla^je^{(t-r
	)\mathcal{A}}\mathcal{F}(\vecc{\Psi})(r)\right\Vert_{\mathscr{E}_1} \textup{d}r\lesssim& \,\begin{multlined}[t]
\int_0^{t/2} (1+t-r)^{-n/4-j/2}\Vert \mathcal{F}(\vecc{\Psi})(r)\Vert_{L^1} \,\textup{d}r\\
+ \int_0^{t/2} e^{-(t-r)}\Vert  \nabla^j\mathcal{F}(\vecc{\Psi})(r)\Vert_{L^2}\, \textup{d}r, \end{multlined} 
\end{aligned}
\end{equation}
where $\mathcal{F}$ is defined as in \eqref{def_F}. We then have by employing H\"older's inequality,
\begin{equation}\label{F_L_1_Estimate_0}
\begin{aligned}
\Vert \mathcal{F}(\vecc{\Psi})(t)\Vert_{L^1}\lesssim\, \Vert vw\Vert_{L^1}
\lesssim& \, \Vert v\Vert_{L^2} \Vert w\Vert_{L^2}
\lesssim \,
 \Vert v\Vert_{L^2} ^2+ \Vert w\Vert^2_{L^2}.
 \end{aligned}
\end{equation}
By using the above estimate and recalling the definition of $\mathcal{M}$ in \eqref{M_t_Decay}, we have
 \begin{equation}
 \begin{aligned}
&\int_0^{t/2} (1+t-r)^{-n/4-j/2}\Vert \mathcal{F}(\vecc{\Psi})(r)\Vert_{L^1} \,\textup{d}r\\
\lesssim& \,\mathcal{M}^2(t)\int_0^{t/2} (1+t-r)^{-n/4-j/2}(1+r)^{-n/2} \textup{d} r\\
\lesssim&\,  \mathcal{M}^2(t)\int_0^{t/2} (1+t)^{-n/4-j/2}(1+r)^{-n/2} \textup{d} r.
\end{aligned}
\end{equation}
We can further bound the integral on the right, leading to
 \begin{equation}\label{J_1_1}
\begin{aligned}
&\int_0^{t/2} (1+t-r)^{-n/4-j/2}\Vert \mathcal{F}(\vecc{\Psi})(r)\Vert_{L^1} \,\textup{d}r\\
\lesssim&\,  \mathcal{M}^2(t)(1+t)^{-n/4-j/2} \int_0^{t/2} (1+r)^{-n/2} \textup{d}r
\lesssim\,
\mathcal{M}^2(t)(1+t)^{-n/4-j/2},
\end{aligned}
\end{equation}
because $n>2$. To estimate $\int_0^{t/2} e^{-(t-r)}\Vert  \nabla^j\mathcal{F}(\vecc{\Psi})(r)\Vert_{L^2}\, \textup{d}r$, we distinguish the cases $j=0$ and $j=1$. First for $j=0$, we have 
\begin{equation}\label{Estimate_j_0}
\begin{aligned}
\Vert  \mathcal{F}(\vecc{\Psi})(t)\Vert_{L^2}
\lesssim&\, \Vert v\Vert_{L^\infty} \Vert w\Vert_{L^2} 
\lesssim\,  M_0(t)(1+t)^{-3n/8}\mathcal{M}(t)(1+t)^{-n/4-1/2}\\
\lesssim&\,  M_0(t)\mathcal{M}(t)(1+t)^{-5n/8-1/2}
\lesssim\,  M_0(t)\mathcal{M}(t)(1+t)^{-3n/4},
\end{aligned}
\end{equation}
because $n \leq 4$. For $j=1$, we have by using \eqref{Ladyz_Ineq} and \eqref{First_inequaliy_Guass} 
\begin{equation} \label{F_L_1_Estimate}
\begin{aligned}
\Vert \nabla \mathcal{F}(\vecc{\Psi})(t)\Vert_{L^2}\lesssim&\, \Vert \nabla(vw)\Vert_{L^2}\\ 
\lesssim&\,   \|\nabla v\|_{L^4}\|w\|_{L^4}+\|v\|_{L^\infty}\|\nabla w\|_{L^2} \\
\lesssim&\, \Vert \nabla v\Vert_{L^2}^{1-n/4}\Vert \nabla^2 v \Vert_{L^2}^{n/4} \Vert w\Vert_{L^2}^{1-n/4}\Vert \nabla w\Vert_{L^2}^{n/4}+\|v\|_{L^\infty}\|\nabla w\|_{L^2}\\
\lesssim&\, \Vert \nabla v\Vert_{L^2}^{1-n/4}\Vert \nabla \vecc{U} \Vert_{L^2}^{n/4} \Vert w\Vert_{L^2}^{1-n/4}\Vert \nabla w\Vert_{L^2}^{n/4}+\|v\|_{L^\infty}\|\nabla w\|_{L^2}.
 \end{aligned}
\end{equation}
Keeping in mind how $M_0$ is defined in \eqref{M_0_t}, we have from above 
\begin{equation}\label{Estimate_First_Term}
\begin{aligned}
\Vert \nabla \mathcal{F}(\vecc{\Psi})(t)\Vert_{L^2}\lesssim&\, (1+t)^{-5n/8-1/2} \mathcal{M}^2(t)+M_0(t)\mathcal{M}(t)(1+t)^{-5n/8-1/2}\\
\lesssim& \, (1+t)^{-5n/8-1/2}(\mathcal{M}^2(t)+M_0(t)\mathcal{M}(t)).
\end{aligned}
\end{equation}
Consequently, by combining the above bound with \eqref{Estimate_j_0}, we deduce  
\begin{equation}\label{J_1_2}
\int_0^{t/2} e^{-(t-r)}\Vert  \nabla^j\mathcal{F}(\vecc{\Psi})(r)\Vert_{L^2}\, \textup{d}r \lesssim \,
(1+t)^{-5n/8-j/2}(\mathcal{M}^2(t)+M_0(t)\mathcal{M}(t))
\end{equation}
for $j\in \{0,1\}$. The integral $\int_{t/2}^{t} \left\Vert \nabla^j e^{(t-r
	)\mathcal{A}}\mathcal{F}(\vecc{\Psi})(r)\right\Vert_{\mathscr{E}_1} \textup{d}r$ is estimated by applying the linear decay rate given in \eqref{Decay_Linearized}
with $j=1$, but using $\mathcal{F}(\vecc{\Psi})(r)$ instead of $%
\vecc{U}_{0}$. By doing so, we obtain 
\begin{equation}
\begin{aligned}
\int_{t/2}^{t} \left\Vert \nabla^j e^{(t-r
	)\mathcal{A}}\mathcal{F}(\vecc{\Psi})(r)\right\Vert_{\mathscr{E}_1} \textup{d}r=&\,\int_{t/2}^{t}\left\Vert \nabla e^{(t-r)\mathcal{A}}
\mathcal{F}(\vecc{\Psi})(r)\right\Vert _{L^{2}}\, \textup{d}r \\
\lesssim&\, \begin{multlined}[t]\int_{t/2}^{t}\left( 1+t-r\right) ^{-\frac{n}{4}-\frac{1}{2}%
}\left\Vert \mathcal{F}(\vecc{\Psi})(r)\right\Vert _{L^{1}}\, \textup{d}r \\
+\int_{t/2}^{t}e^{-\lambda\left( t-r\right) /2}\left\Vert \nabla \mathcal{F}%
(\vecc{\Psi})(r)\right\Vert _{L^{2}}\, \textup{d}r. 
\end{multlined}
\end{aligned}
\end{equation}%
On the other hand, we have by applying \eqref{F_L_1_Estimate_0} and recalling the definition of $\mathcal{M}$ in \eqref{M_t_Decay},  
\begin{equation}
\Vert \mathcal{F}(\vecc{\Psi})(t)\Vert _{L^{1}(\mathbb{R}^{n})}\lesssim  \, \mathcal{M}^{2}(t)(1+t)^{-n/2}.
\end{equation}%
Thus, we can derive the bound
\begin{equation}
\begin{aligned}
\int_{t/2}^{t}\left( 1+t-r\right) ^{-\frac{n}{4}-\frac{1}{2}%
}\left\Vert \mathcal{F}(\vecc{\Psi})(r)\right\Vert _{L^{1}}\, \textup{d}r
 \lesssim&\, \mathcal{M}^{2}(t)\int_{t/2}^{t}\left( 1+t-r\right) ^{-\frac{n%
}{4}-\frac{1}{2}}(1+r)^{-\frac{n}{2}}\textup{d}r   \\
\lesssim&\, \mathcal{M}^{2}(t)(1+t/2)^{-\frac{n}{2}%
}\int_{t/2}^{t}\left( 1+t-r\right) ^{-\frac{n}{4}-\frac{1}{2}}\textup{d}r.  
\end{aligned}
\end{equation}%
Because $n>2$, then we know that
\begin{equation} \label{J_21_estimate}
\begin{aligned}
&\int_{t/2}^{t}\left( 1+t-r\right) ^{-\frac{n}{4}-\frac{1}{2}%
}\left\Vert \mathcal{F}(\vecc{\Psi})(r)\right\Vert _{L^{1}}\, \textup{d}r\\
\lesssim&\, \mathcal{M}^{2}(t)(1+t/2)^{-\frac{n}{2}}\int_{0}^{t/2}\left( 1+r\right)
^{-\frac{n}{4}-\frac{1}{2}}\textup{d}r   
\lesssim\, \mathcal{M}^{2}(t)
(1+t)^{-\frac{n}{4}-\frac{1}{2}}.
\end{aligned}
\end{equation}%
Furthermore, we have by using the bound \eqref{Estimate_First_Term} that
\begin{equation}\label{J_2_2_Estimate}
\int_{t/2}^{t}e^{-\lambda\left( t-r\right) /2}\left\Vert \nabla \mathcal{F}%
(\vecc{\Psi})(r)\right\Vert _{L^{2}}\, \textup{d}r\lesssim (1+t)^{-5n/8-1/2}(\mathcal{M}^2(t)+M_0(t)\mathcal{M}(t)).
\end{equation}%
Therefore, by combining estimates \eqref{Decay_Liearized_Prob}, \eqref{J_1_1}, \eqref{J_1_2}, \eqref{J_21_estimate}, and the above inequality, we infer    
\begin{equation} \label{Estimate_Decay_2}
\begin{aligned}
\Vert \nabla ^{j}\vecc{U}(t)\Vert _{L^2} \lesssim& \, \begin{multlined}[t]
(1+t)^{-n/4-j/2}\left( \Vert \vecc{U}_{0}\Vert _{L^{1}}+\Vert \nabla ^{j}\vecc{U}_{0}\Vert _{L^{2}}\right) 
  \\
+\mathcal{M}^{2}(t)(1+t)^{-n/4-j/2}+(1+t)^{-n/4-1/2-j/2}M_{0}(t)%
\mathcal{M}(t)\end{multlined}
\end{aligned}
\end{equation} 
for $n=3$ and $j \in \{0,1\}$. At this point we also need an estimate of $\|w\|_{L^2}$. Recalling the energy bound we obtained in \eqref{E_0_Energy}, we have
\begin{equation}
\begin{aligned}
	\frac{1}{2}\frac{\textup{d}}{\dt}\int_{\mathbb{R}^{n}}\tau \left\vert  
	w\right\vert ^{2}\dx+\frac{1}{2}\int_{\mathbb{R}^{n}} |w|^2 \dx
	\lesssim\,
	 \Vert \nabla\vecc{\Psi}(t)\Vert_{\mathscr{E}_1}^2	+|R^{(1)}(w)|. 
	\end{aligned}
\end{equation}
By applying Gronwall's inequality, we deduce that 
\begin{equation} \label{wL2}
\begin{aligned}
\Vert w(t)\Vert_{L^2}^2\lesssim& \, \begin{multlined}[t]\Vert w_0\Vert_{L^2}^2\exp{(-\tfrac{1}{\tau} t)}+\int_{0}^t \Vert \nabla\vecc{\Psi}(t)\Vert_{\mathscr{E}_1}^2 	\exp{(-\tfrac{1}{\tau}(t-s))}\, \textup{d}s\\  
+\int_{0}^t |R^{(1)}(w)(s)|)\exp{(-\tfrac{1}{\tau}(t-s))}\, \textup{d}s.\end{multlined}
\end{aligned}
\end{equation}
We need to further estimate the two integrals on the right. By making use of the bound \eqref{Estimate_Decay_2} with $j=1$, we have 
\begin{equation}
\begin{aligned}
&\int_{0}^t  \Vert \nabla^j\vecc{\Psi}(t)\Vert_{\mathscr{E}_1}^2	\exp{(-\tfrac{1}{\tau}(t-s))}\, \textup{d}s\\ \lesssim&\, \begin{multlined}[t](1+t)^{-n/4-1/2}\left( \Vert \vecc{U}_{0}\Vert _{L^{1}}+\Vert \nabla \vecc{U}_{0}\Vert _{L^{2}}\right)\\
+\mathcal{M}^{2}(t)(1+t)^{-n/4-1/2}+(1+t)^{-n/4-1}M_{0}(t)%
\mathcal{M}(t). \end{multlined}
\end{aligned}
\end{equation}
Concerning the second integral on the right in \eqref{wL2}, we find
\begin{equation}
\begin{aligned}
|R^{(1)}(w)(t)|\lesssim\,\Vert v w^2\Vert_{L^1}
\lesssim \, \Vert v\Vert_{L^\infty}\Vert w\Vert_{L^2}^2
\lesssim& \, M_0(t)\mathcal{M}^2(t) (1+t)^{-3n/8} (1+t)^{-n/2-1}\\[1mm]
\lesssim& \,  M_0(t)\mathcal{M}^2(t) (1+t)^{-(\frac{7n}{8}+1)}. 
\end{aligned}
\end{equation}
This inequality immediately yields 
\begin{eqnarray*}
\int_{0}^t |R^{(1)}(w)(s)|)\exp{(-\tfrac{1}{\tau}(t-s))}\, \textup{d}s\lesssim  M_0(t)\mathcal{M}^2(t) (1+t)^{-(\frac{7n}{8}+1)}. 
\end{eqnarray*}
Consequently, we deduce from above that 
\begin{equation}
\begin{aligned}
\Vert w(t)\Vert_{L^2}\lesssim&\, \begin{multlined}[t] \Vert w_0\Vert_{L^2}\exp{(-\tfrac{1}{2\tau} t)}+(1+t)^{-n/4-1/2}\left( \Vert \vecc{U}_{0}\Vert _{L^{1}}+\Vert \nabla \vecc{U}_{0}\Vert _{L^{2}}\right)\\[1mm]
+\mathcal{M}^{2}(t)(1+t)^{-n/4-1/2}+(1+t)^{-n/4-1}M_{0}(t)%
\mathcal{M}(t) \\[1mm]
+\sqrt{M_0(t)} \mathcal{M}(t)  (1+t)^{-(\frac{7n}{16}+\frac{1}{2})},\end{multlined}
\end{aligned}
\end{equation}
which further implies that
\begin{equation}\label{w_Estimate_Nonl}
\begin{aligned}
\Vert w(t)\Vert_{L^2} \lesssim& \, \begin{multlined}[t]\left(\Vert w_0\Vert_{L^2}+ \Vert \vecc{U}_{0}\Vert _{L^{1}}+\Vert \nabla \vecc{U}_{0}\Vert _{L^{2}}\right)(1+t)^{-n/4-1/2}\\
 \Big(\mathcal{M}^{2}(t)+\sqrt{M_0(t)} \mathcal{M}(t) +M_{0}(t)%
\mathcal{M}(t) \Big)(1+t)^{-n/4-1/2}. \end{multlined}
\end{aligned}
\end{equation}
By also using the fact that
$
\Vert  v\Vert_{L^2}\lesssim\Vert  w\Vert_{L^2}+\Vert  \vecc{U}\Vert_{L^2},  
$
together with estimates \eqref{Estimate_Decay_2} and \eqref{w_Estimate_Nonl}, we obtain
\begin{equation}\label{v_estimate_Nonl}
\begin{aligned}
\Vert  v(t)\Vert_{L^2}\lesssim&\, \begin{multlined}[t]\left(\Vert w_0\Vert_{L^2}+ \Vert \vecc{U}_{0}\Vert _{L^{1}}+\Vert  \vecc{U}_{0}\Vert _{H^{1}}\right)(1+t)^{-n/4}\\
+\Big(\mathcal{M}^{2}(t)+\sqrt{M_0(t)} \mathcal{M}(t) +M_{0}(t)%
\mathcal{M}(t) \Big)(1+t)^{-n/4}. \end{multlined}
\end{aligned}
\end{equation}
By collecting \eqref{Estimate_Decay_2}, \eqref{w_Estimate_Nonl} and \eqref{v_estimate_Nonl} and recalling the definition of $\mathcal{M}(t)$ in \eqref{M_t_Decay}, we find 
\begin{equation}
\begin{aligned}
\mathcal{M}(t)\lesssim \, \Vert w_0\Vert_{L^2}+ \Vert \vecc{U}_{0}\Vert _{L^{1}}+\Vert  \vecc{U}_{0}\Vert _{H^1}
+\mathcal{M}^{2}(t)+\sqrt{M_0(t)} \mathcal{M}(t) +M_{0}(t)%
\mathcal{M}(t).  
\end{aligned}
\end{equation}
By relying on \eqref{M_0_M_Ineq}, we deduce that 
\begin{eqnarray*}
\mathcal{M}(t)\lesssim \, \Vert w_0\Vert_{L^2}+ \Vert \vecc{U}_{0}\Vert _{L^{1}}+\Vert \nabla \vecc{U}_{0}\Vert _{L^{2}}+\mathcal{M}^{3/2}(t)+\mathcal{M}^2(t).
\end{eqnarray*}
This last estimate together with Lemma~\ref{Lemma_Stauss} implies that there exists $C>0$, independent of time, such that
\begin{eqnarray*} 
\mathcal{M}(t) \, \lesssim \Lambda_0,
\end{eqnarray*}  
provided that $\Lambda_0=\Vert w_0\Vert_{L^2}+ \Vert \vecc{U}_{0}\Vert _{L^{1}}+\Vert \nabla \vecc{U}_{0}\Vert _{L^{2}}$ is small enough. This step completes the proof of Theorem \ref{Theorem_Decay}. 
\end{proof}
\begin{remark}[On optimality of the decay rates]
It is known that in the absence of the memory term (i.e., when $g=0$), the same decay rates as in \eqref{decay_rates} are optimal for the linearized problem;   see~\cite{PellSaid_2019_1}. For small enough data, the nonlinear problem's solution is expected to behave in the same manner.  \\
\indent It has been proven that  adding the memory damping to the linear  wave damped equation:
\begin{equation}\label{damped_Wave}
u_{tt}-\Delta u+u_t=0,\quad x\in \R^n,\, t>0
\end{equation}
will not improve the decay rate even if the kernel decays exponentially; see~~\cite{dharmawardane2010decay}. Moreover, wave equation with an exponentially decaying memory 
(i.e., the term $u_t$ replaced by $\int_0^t g(s)\Delta u(t-s)ds$ in \eqref{damped_Wave}) has the same decay rate as with the linear damping only; see~\cite{conti2007decay, matsumura1977global}. This suggests that combining linear damping with the memory of type $I$ leads to saturation. We thus expect the estimates of   Theorem~\ref{Theorem_Decay} to be sharp. A possible first step in proving optimality would be to take the memory kernel in the form $g(t)=e^{-\kappa t}, \, \kappa>0$ and apply the Fourier transform in $x$ and the Laplace transform with respect to $t$ to the linearized problem. However, further     challenging questions arise about inverting the Laplace transform and relating the linear result to the nonlinear problem, and are thus left for future work.
\end{remark} 

\bibliography{references}{}
\bibliographystyle{siam} 
\end{document}